\documentclass[11pt]{article}
\usepackage{etex}

\usepackage[margin=1in]{geometry}
\usepackage{stackengine}

\usepackage{stmaryrd}

\usepackage{color}
\usepackage[latin1]{inputenc}
\usepackage[T1]{fontenc}
\usepackage[normalem]{ulem}
\usepackage[english]{babel}
\usepackage{verbatim}
\usepackage{graphicx}
\usepackage{enumerate}
\usepackage{amsmath,amssymb,amsfonts,amsthm,amscd,mathrsfs}
\usepackage{array}

\usepackage{dsfont}

\usepackage[T1]{fontenc}

\usepackage{url}

\usepackage{caption}
\usepackage{subcaption}
\usepackage[bookmarksopen, bookmarksnumbered]{hyperref}
\usepackage{tikz}
\usetikzlibrary{arrows}
\usetikzlibrary{arrows.meta}
\definecolor{wwhhii}{rgb}{1.,1.,1.}
\definecolor{rreedd}{rgb}{1.,0.,0.}
\definecolor{uuuuuu}{rgb}{0.26666666666666666,0.26666666666666666,0.26666666666666666}
\definecolor{darkgreen}{HTML}{0d8513}

\newtheorem{theorem}{Theorem}[section]

\newtheorem{lemma}[theorem]{Lemma}
\newtheorem{cla}[theorem]{Claim}
\newtheorem{prop}[theorem]{Proposition}

\theoremstyle{definition}

\input xy
\xyoption{all}

\DeclareMathOperator{\argmax}{argmax}
\DeclareMathOperator{\Exp}{Exp}
\DeclareMathOperator{\Var}{Var}

\DeclareMathOperator{\Graph}{Graph}
\DeclareMathOperator{\hypo}{hypo}

\newcommand{\NTP}{\mathsf{NoTP}}

\newcommand{\E}{\mathbb E}
\newcommand{\PP}{\mathbb P}
\newcommand{\Z}{\mathbb Z}
\newcommand{\R}{\mathbb R}

\newcommand{\N}{\mathbb N}

\newcommand{\LL}{\mathbb L}

\newcommand{\don}{\mathds{1}}

\newcommand{\cC}{\mathcal C}

\newcommand{\cK}{\mathcal K}

\newcommand{\cB}{\mathcal B}

\newcommand{\cE}{\mathcal E}
\newcommand{\cA}{\mathcal A}

\newcommand{\cL}{\mathcal L}

\newcommand{\ocE}{\overline{\cE}}

\newcommand{\cH}{\mathcal H}

\newcommand{\cU}{\mathcal U}

\newcommand{\sF}{\mathscr F}

\newcommand{\LE}{\mathscr L}

\newcommand{\fG}{\mathsf R}
\newcommand{\sfS}{\mathsf S}

\newcommand{\fr}{\mathsf r}

\newcommand{\bc}{\mathbf c}

\newcommand{\boo}{\mathbf 0}
\newcommand{\bn}{\mathbf n}

\makeatletter
\renewcommand\tableofcontents{
  \null\hfill\textbf{\Large\contentsname}\hfill\null\par
  \@mkboth{\MakeUppercase\contentsname}{\MakeUppercase\contentsname}
  \@starttoc{toc}
}

\g@addto@macro\normalsize{
  \setlength\abovedisplayskip{5pt}
  \setlength\belowdisplayskip{5pt}
  \setlength\abovedisplayshortskip{3pt}
  \setlength\belowdisplayshortskip{3pt}
}

\makeatother

\numberwithin{equation}{section}

\begin{document}
\title{Discrete geodesic local time converges under KPZ scaling}
\author{Shirshendu Ganguly
\thanks{Department of Statistics, UC Berkeley, Berkeley, CA, USA. e-mail: sganguly@berkeley.edu}
\and
Lingfu Zhang
\thanks{Department of Statistics, UC Berkeley, Berkeley, CA, USA. e-mail: lfzhang@berkeley.edu}
}
\date{}

\maketitle

\begin{abstract}
The directed landscape constructed in \cite{DOV} produces a directed, planar, random geometry, and is believed to be the universal scaling limit of two-dimensional first and last passage percolation models in the Kardar-Parisi-Zhang (KPZ) universality class. 
Geodesics in this random geometry form an important class of random continuous curves exhibiting fluctuation theory quite different from that of Brownian motion. In this vein, counterpart to Brownian local time, BLT (a self-similar measure supported on the set of zeros of Brownian motion), a local time for geodesics, GLT, was recently constructed and used to study fractal properties of the directed landscape in \cite{GZ+}.
It is a classical fact and can be proven using the Markovian property of Brownian motion that the uniform discrete measure on the set of zeros of the simple random walk converges to BLT. 
In this paper, we prove the ``KPZ analog'' of this by showing that the local times for discrete geodesics in pre-limiting integrable last passage percolation models converge to GLT under suitable scaling guided by KPZ exponents.

In absence of any Markovianity, our arguments rely on the recently proven convergence of geodesics in pre-limiting models to that in the directed landscape \cite{DV21}.
However, this input concerns macroscopic properties and is too coarse to capture the microscopic information required for local time analysis.  To relate the macroscopic and microscopic behavior,  a key ingredient is an \emph{a priori} smoothness estimate of the local time in the discrete model, proved relying on geometric ideas such as the coalescence of geodesics as well as their stability  under perturbations of boundary data.
\end{abstract}

\setcounter{tocdepth}{2}
\tableofcontents

\section{Introduction and main results} 
The zero set of simple random walks is an object whose study is as old as that of the random walk itself. Estimates on the random walk hitting zero are crucial in understanding the properties of random walk bridges and excursions. The Brownian counterpart gives rise to a canonical fractal subset $Z$ of the real line analogous to the Cantor set but random. The Hausdorff dimension of $Z$ is known to be $1/2$ stemming from the fact that the probability a simple random walk is at $0$ at time $n$ decays as $\frac{1}{\sqrt{n}}.$ There is a natural measure supported on $Z$ known as the Brownian local time (BLT) which can be obtained as the density of the occupation measure of Brownian motion (This is the measure $\mu$ on $\R$ such that $\mu(A)=\int_{0}^1 \mathbf{1}(\mathbb{B}_s\in A)\mathrm{d}s$) at zero. 

While Brownian motion is a canonical one-dimensional random curve, this article focuses on the local time of another class of random curves arising from models of random geometry. A paradigmatic example is what is known as First passage percolation (FPP) where one considers a distortion of the usual graph metric on $\Z^2$
by putting i.i.d.\ random non-negative weights on the edges and considering the corresponding random weighted path metric. A natural class of random curves arising in this situation is formed by the weight-minimizing paths between points; namely the geodesics in the associated metric space. 

The geodesic behavior, encoding a lot of information 
about the underlying geometry, has been an important 
object of investigation both in the rigorous and non-
rigorous literature going back to the 1960s  when the 
model of FPP was first proposed by Hammersley and Welsh. 

However, despite the interest,  the innate hardness of the model has made it very hard to verify rigorously any of these predictions. 
Nonetheless, an important related class of models for which some progress has been made is
known as Last passage percolation (LPP). Ignoring 
microscopic specifications, in general terms they all 
consist of an environment of random noise
through which \emph{directed} paths travel, accruing the 
integral of the noise along it --- a quantity known as 
energy or weight. Given two endpoints, a maximization is done 
over the weights of all paths with these endpoints with the 
optimizing path still being called a \emph{geodesic}, despite its maximizing nature as opposed to minimizing as in FPP.

Remarkably, particular choices of the underlying noise
environment make LPP ``integrable'', admitting certain 
exact algebraic properties such as distributional identities 
involving eigenvalues of random matrices, young 
diagrams, and so on, which provide key inputs for 
probabilistic analysis. 

Both FPP and LPP models are expected to belong to the 
 Kardar-Parisi-Zhang (KPZ) universality class which 
 encompasses  a broad family of models of one-dimensional random growth 
which are believed to exhibit certain common features, such as local growth driven by white noise and a 
smoothing effect from some notion of surface tension and
a slope-dependent non-linear growth.
 It is predicted  that the fluctuation theories of such models are dictated by universal scaling exponents and limiting distributions, 
exhibited by solutions to a stochastic PDE known as the KPZ equation \cite{Corwin12,Corwin16,HQ18}.
For these $(1+1)$-dimensional models, the resulting
evolution of the growth interface $h(t,x)$ is described by 
 the characteristic triple $(1,\frac{1}{3},\frac{2}{3})$ of 
exponents: at time $t$, the value of $h(t,x)$ is of order
$t^1$, with order $t^{1/3}$ deviations from its mean, and 
further, non-trivial correlations are observed when $x$ is
 varied on the scale of $t^{2/3}$.
Furthermore, $h(t,x)$ on proper centering and scaling  
according to these exponents, is expected to converge to a universal limit $t\to\infty$ \cite{QR14, MQR}. 
 In terms of geodesics in, say, LPP, this corresponds to having a weight fluctuation exponent $1/3$ and a 
 transversal fluctuation exponent (deviation from the 
 straight line joining its endpoints) $2/3$.

Another key feature in such random settings is the
 phenomenon of {\it{geodesic coalescence}} where 
 geodesics between different points  tend to merge 
 together and share a significant fraction of their journey,
  passing through the atypically high values of the 
  underlying noise variables.  Note that this is quite 
  different from standard Euclidean geometry where 
  geodesics are merely straight lines and hence have 
  trivial intersection patterns.

Owing to such coalescence properties, even though the geodesic is a global object, with high probability, a local segment of it is likely to depend only on the nearby noise. In other words, it exhibits a form of weak decay of correlation property.

In \cite{DOV} the directed landscape was constructed. 
This is a random continuum metric space expected to be
 the universal scaling limit of a class of FPP and LPP
models. Moreover, in a follow-up work \cite{DV21}, a class
of exactly solvable LPP models was indeed shown to
converge to the landscape. The latter also admits geodesics (a non-trivial fact) and it was further shown that 
the pre-limiting geodesics converge to the limiting one in a certain topology.    

However, the topology of convergence is not refined enough to conclude the convergence of several observables of interest associated with the geodesics to their limiting counterparts. 

In this article, we focus on the convergence of one such observable, namely that of the discrete local time, i.e., the amount of time a pre-limiting discrete geodesic intersects a line.
While defining the discrete local time is straightforward (the formal definition appears later in \eqref{dlt}), its analog for the directed landscape is rather non-trivial and was one of the key accomplishments in \cite{GZ+}  by the same authors. The notion of local time turned out to be a  crucial ingredient in the analysis of the exceptional set of points witnessing geodesic non-coalescence in the above-mentioned paper. 
 
Thus this paper along with \cite{GZ+} initiates the study of local times in the context of geodesics, akin to the study of the same for random walks and Brownian motion which  can be considered classical, and hence forms a part of a bigger program of investigating the similarities and dissimilarities between the geodesic and Brownian motion. The arguments in this paper can be seen as one exploiting the weak decay of correlation the geodesic exhibits as alluded to above --- a style of argument that we expect would be useful in several other applications going forward.  We now move on to the statement of our results first by 
introducing the key objects in play.

\subsection{Directed landscape}
The directed landscape $\cL$ is a random continuous function from the parameter space 
\[
\R^4_\uparrow = \{u = (p; q) = (x, s; y, t) \in \R^4 : s < t\}
\]
to $\R$ constructed in the breakthrough work \cite{DOV}. 
It satisfies the following composition law 
\begin{equation}  \label{eq:DL-compo}
\cL(x,r;z,t) = \max_{y\in \R}     \cL(x,r;y,s) + \cL(y,s;z,t)
\end{equation}
for any $x,z \in \R$ and $r<s<t$.
For a continuous path $\pi:[s,t] \to \R$, its weight under the metric $\cL$ is given by
\[
\|\pi\|_\cL = \inf_{k\in \N} \inf_{s=t_0< t_1 < \cdots < t_k = t} \sum_{i=1}^k \cL(\pi(t_{i-1}), t_{i-1}; \pi(t_i), t_i).
\]
Such a path $\pi$ is said to be a geodesic from $(\pi(s), s)$ to $(\pi(t), t)$ if $\|\pi\|_\cL = \cL(\pi(s), s; \pi(t), t)$.
It is shown in \cite{DOV} that almost surely, there exists at least one geodesic from $p$ to $q$ for any $u = (p; q) = (x, s; y, t) \in \R^4_\uparrow$.
We let $\pi_u = \pi_{(p;q)}=\pi_{(x,s;y,t)}:[s,t]\to\R$ denote any such geodesic from $p$ to $q$.
Further, for any fixed $u\in \R^4_\uparrow$, almost surely the geodesic $\pi_u$ is unique.
{Towards establishing its universal nature, in \cite{DOV,DV21}, the directed landscape $\cL$ and the associated geodesics are proved to be the scaling limits of their counterparts for various exactly solvable LPP models, including Brownian LPP on $\R\times \Z$, Poissonian LPP on $\R^2$, and LPP on $\Z^2$ with i.i.d.\ Exponential or Geometric weights (Exponential/Geometric LPP).}\\

\noindent
\textbf{Semi-infinite geodesics:} Going beyond point-to-point geodesics, 
for each $p=(x,s) \in \R^2$ and $r\in \R$, we denote by $\pi_p^r:[s,\infty)\to \R$ a semi-infinite geodesic started from $p$ in the $r$ direction; i.e. we have $t^{-1}\pi_p^r(t) \to r$ as $t\to\infty$, and for any $s\le s_1 < s_2 <\infty$, the restriction of $\pi_p^r$ on $[s_1,s_2]$ is a geodesic (from $(\pi_p^r(s_1), s_1)$ to $(\pi_p^r(s_2), s_2)$).
When $r=0$, we will drop the $r$ dependence and simply use $\pi_p=\pi_p^0$. 
While the issues regarding the existence and uniqueness of such semi-infinite geodesics are not straightforward, they were treated in a systematic way in \cite[Section 4]{GZ+} and \cite{RV21,BSS22}, with the following established:
almost surely, such semi-infinite geodesic $\pi_p$ exists for any $p\in \R^2$;
and for fixed $p$, almost surely there is a unique semi-infinite geodesic.

The landscape and thereby the semi-infinite geodesic $\pi_{(0,0)}$ admits invariance under appropriate scalings
guided by the exponent triple $(1,1/3,2/3)$, characteristic 
of the KPZ universality class, much like Brownian motion 
which is invariant under diffusive scaling. Hence, as in the 
case of the latter, both $\cL$ and $\pi_{(0,0)}$ are expected to exhibit various fractal or self-similar properties.   \\

\subsection{Pre-limiting model: Exponential LPP} On the pre-limiting side, our choice of the model is Exponential LPP\footnote{Our arguments would work equally well for Geometric LPP, the other known integrable LPP model on $\Z^2$. Two other known integrable models are Poissonian LPP on $\R^2$ and Brownian LPP on $\R\times \Z$, and the arguments in this paper also imply similar conclusions for the natural notions of local times in these settings. A longer discussion appears following Theorem \ref{thm:converge-lc}.}.  
Consider directed last passage percolation (LPP) on $\Z^2$ with i.i.d.\ Exponential weights on the vertices, i.e., we have a random field $\{\omega_p: p\in \Z^2\}$ where $\omega_p$ are i.i.d.\ $\Exp(1)$ random variables. For any up/right path $\gamma$ from $p$ to $q$ where $p\preceq q$ (i.e., $p$ is co-ordinate wise smaller than or equal to $q$)  the weight of $\gamma$ is defined by 
\[T(\gamma):=\sum_{w\in \gamma} \omega_{w}.\] (Note that owing to the continuity of the Exponential random variables, almost surely, $T(\gamma_1)\neq T(\gamma_2)$ for $\gamma_1\neq \gamma_2$).
For any two points $p$ and $q$ with $p\preceq q$, we shall denote by $T_{p,q}$ the last passage time from $p$ to $q$; i.e., 
\[T_{p,q}=\max_\gamma T(\gamma),\]
where the maximum is over all directed paths from $p$ to $q$.
The (almost surely unique) directed path with the maximum weight is the geodesic from $p$ to $q$, denoted as $\Gamma_{p,q}$.
When $p\not\preceq q$, we let $T_{p,q}=-\infty$.

In the setting of Exponential LPP, the existence and uniqueness of semi-infinite geodesics have been established (see \cite{Cou11, FP05}), and we summarize the known facts here.
For any $p\in\Z^2$, there is an almost surely unique upper-right path from it, denoted as $\Gamma_p=\{p=(x_0, y_0), (x_1,y_1), \cdots \}$, such that for any $0\le i \le j$, the path $(x_i,y_i), \cdots, (x_j,y_j)$ is $\Gamma_{(x_i,y_i), (x_j,y_j)}$, and $\lim_{i\to\infty} \frac{x_i}{y_i}=1$.\\
\noindent\textbf{Discrete semi-infinite geodesic:} Such a $\Gamma_p$ is referred to as the $(1,1)$-direction semi-infinite geodesic from $p$.

Almost surely, for any sequence $q_n\in\Z^2$ that goes to infinity in the $(1,1)$-direction, $\Gamma_{p,q_n}$ converges to $\Gamma_p$ in the sense of local limits (i.e., any initial segment of $\Gamma_{p}$ is also the initial segment of $\Gamma_{p,q_n}$ for all large $n$).
Also for any $p, q\in \Z^2$, the paths $\Gamma_p$ and $\Gamma_q$ are the same except for finitely many points, i.e., they will eventually merge and stay together after finitely many steps.
Thus all the semi-infinite geodesics in the $(1,1)$-direction form a tree structure.

Throughout we shall always assume (the probability $1$ event) that for any $p\preceq q\in \Z^2$, there is a unique geodesic $\Gamma_{p,q}$,
and a unique $(1,1)$-direction semi-infinite geodesic $\Gamma_p$.

Further, for the rest of this paper, we let $d, ad:\Z^2\to \Z$ be the functions $d(x,y)=x+y$ and $ad(x,y)=x-y$. For each $n\in\Z$, we let $\LL_n=\{p\in\Z^2: d(p)=2n\}$.
Given that the orientation of the directed landscape is vertically up, for comparison purposes it might be more natural to rotate the Exponential LPP model by $\pi/4$ so that the vertical and horizontal directions are the temporal and spatial directions respectively. However, we do not adopt such a change and stick to the traditional definition so as to be consistent with the previous literature on Exponential LPP,  which will be quoted several times in this paper.

\subsection{Main result}

We first recall the local time for the semi-infinite geodesic $\pi_{(0,0)}(t)$ in the directed landscape {from \cite[Section 7.1]{GZ+}} where it was introduced and constructed.
We start by considering the occupation measure by defining for any $w>0$ and $h\ge 0$,
\[
L_w(h):= \int_0^h \don[\pi_{(0,0)}(t)\in [-w,w]] dt.
\]
The following proposition defines the local time as the density of the occupation measure at $0.$
\begin{prop}[\protect{\cite[Proposition 2.4, Proposition 7.4]{GZ+}}]  \label{prop:om-limit}
Almost surely there exists a non-decreasing function $L:[0,\infty) \to \R$, such that $L(h)=\lim_{w\to 0} (2w)^{-1}L_w(h)<0$ for any $h\ge 0$.
In addition, there exist constants $c,C>0$ such that for any $0\le g < h$ and $M>0$, we have
\[
\PP[L(h)-L(g)>M(h-g)^{1/3}] < Ce^{-cM}.
\]
\end{prop}

It is clear from the definition that $L$ is supported inside the zero set of $\pi_{(0,0)}$. Further, the above estimate indicates, and it was indeed proved in \cite{GZ+}, that the latter has Hausdorff dimension $1/3.$ This is in contrast to the $1/2$ dimension-ness of the zero set of Brownian motion. The disparity stems from the fact that while Brownian motion is $\frac12^{-}$-H\"older regular, $\pi_{(0,0)}$ is $\frac23^{-}$-H\"older regular \cite[Proposition 12.3]{DOV} (see also \cite{dauvergne2022three} where geodesics are shown to have $3/2$-variation). 

Correspondingly, in Exponential LPP, for each $n\in\N$ define the discrete local time function $L^{(n)}:\R_{\ge 0}\to \R_{\ge 0}$ by 
\begin{equation}\label{dlt}
L^{(n)}(h) = 2^{2/3}n^{-1/3}|\{0\le i \le hn: (i,i) \in \Gamma_{(0,0)} \}|.
\end{equation} 
Here the factor $2^{2/3}n^{-1/3}$ makes $L^{(n)}$ be in the same scale as $L$, dictated by the appearance of the same factors in the convergence of Exponential LPP to the directed landscape (to be recorded shortly in Section \ref{ssec:lpptodl}).
The main result of this paper is the convergence of $L^{(n)}$ to $L$.
\begin{theorem}  \label{thm:converge-lc}
As $n\to\infty$, we have $L^{(n)}\to L$ weakly, in the topology of uniform convergence on compact sets.
\end{theorem}
We remark that our arguments also go through (essentially verbatim) for other exact-solvable LPP models, such as Geometric, Poissonian, and Brownian LPP. For Geometric LPP,  one issue that arises is that of  multiple geodesics between vertices owing to the discreteness of vertex weights. In such a situation one could consider the ``rightmost'' (or ``leftmost'') semi-infinite geodesic from $(0,0)$.
In Poissonian LPP, one considers a Poisson point process on $\R^2$, and the geodesic between two points $p, q \in \R^2$ with $p\preceq q$ (i.e., $p\le q$ in each co-ordinate) is any up/right path between them, passing through the most number of points.
In this setting, the pre-limiting local time can similarly be defined by counting the number of intersections of the ``rightmost'' (or ``leftmost'') semi-infinite geodesic from $(0,0)$ with the diagonal $y=x$.
Finally, in Brownian LPP, the random field is given by a sequence of independent two-sided Brownian motions $\{\cB_i\}_{i\in\Z}$, and for any two points $p, q \in \R\times\Z$ with $p\preceq q$ (i.e., $p\le q$ in each co-ordinate as before), the geodesic between them is the (almost surely unique) up/right path in $\R\times\Z$, which accumulates the largest Brownian motion increments.
In this case, the analogous pre-limiting local time could be taken to be the number of integers $i$ such that $(i,i)$ is on the semi-infinite geodesic from $(0,0)$.
For each of these exactly-solvable LPP models, using minor variations of the arguments in this paper, one can deduce  results analogous to Theorem \ref{thm:converge-lc}.

As another remark on the generality of the methods in this paper, we point out that (by slightly adapting our arguments) similar convergence of local times can also be proved for point-to-point geodesics, going beyond semi-infinite geodesics, with the  limit  being the local time of the corresponding point-to-point geodesic in the directed landscape, which can also be constructed using the same methods employed in \cite{GZ+} for the construction of the semi-infinite geodesic local time.\\

We finish this subsection with a brief overview of the more classical literature on BLT as well as the more recent advances in understanding geodesics models of random geometry which this paper contributes to as well. 
The notion of local time was invented by Paul L\'evy, and BLT was constructed as a continuous process by Trotter \cite{trotter1958property}. 
In terms of convergence from discrete models, it was first shown by Knight in \cite{knight1963random} that the simple symmetric random walk local time converges to BLT. 
Convergence of the local time of more general random walks to BLT was later established in \cite{perkins1982weak, borodin1981asymptotic}.
(See also \cite{skorokhod1965limit, skorokhod1966limit} whose results plus Tanaka's formula also imply convergence to BLT of general random walks, as explained in the introduction of \cite{borodin1986character}. In the same paper \cite{borodin1986character} the difference between the local time of random walks and BLT was carefully analyzed.)

However, perhaps more relevant to our paper is the recent explosion of developments around the understanding of planar models of random geometry in the KPZ universality class (see e.g.\ \cite{balazs2020non,busani2022universality,BHS,BSS,BSS19,dauvergne2022three,dauvergne2021disjoint,janjigian2022geometry,seppalainen2020coalescence,Z20}). As already mentioned above, a key feature of these models is coalescence where geodesics between different points tend to merge with each other. This has naturally led to a recent investigation of atypical behavior of non-coalescence exhibited by a fractal set of points (see e.g. \cite{BGH,BGHhau,GH21,D21last}). In fact, such a recent study in \cite{GZ+} necessitated the construction of the geodesic local time where it was shown that the latter is $\frac{1}{3}^-$ H\"older regular in contrast to $\frac{1}{2}^-$ H\"older regularity for BLT. 

 The consequences of coalescence as well as regularity properties of passage times have found several other applications as well recently. (See \cite{ganguly2021random} for a short survey).

\subsection{Idea of proof}  \label{sec:iop}
Before outlining the key ideas in the proof it might be instructive to discuss an approach to prove the counterpart statement showing the convergence of the random walk local time to the Brownian version. For example, in \cite[Section 6.1]{morters2010brownian}) it is shown that the occupation measure of Brownian motion $\cB$ of the cylinder $[0,1]\times [-w,w]$ is approximated well by the number of times $\cB$ crosses the cylinder, and hence one can equivalently take the limit of the number of such crossings to construct Brownian local time.
The proof of this equivalence relies heavily on the Markovian nature of $\cB$ which provides precise crossing number estimates. The exact embedding of random walk in $\cB$ then allows one to compare the crossing number of the former to the crossing number of the latter, thereby implying convergence.

However, for the geodesic, being a rather global object, there is no Markovianity. 
Nonetheless, owing to coalescence, it does, albeit in a rather implicit sense, exhibit some decay of correlation.  

Our proof relies on making use of this and the known convergence of Exponential LPP and the associated geodesics to their counterparts in the directed landscape (from \cite{DV21}).
However, the topology of convergence is in terms of Hausdorff distance between the geodesics thought of as compact subsets of $\R^2$ which is clearly not refined enough to capture the local time on the line $y=0.$ Nonetheless, it does help relate the occupation measures of the geodesics in the limiting and pre-limiting settings. For instance, the convergence allows us to deduce that the amount of time the ``suitably scaled'' pre-limiting geodesic $\Gamma_{0}$ spends in the cylinder $[-w,w] \times [0,1]$ is close to that spent by $\pi_{(0,0)}$ (relying on an apriori estimate on the occupation measure developed in \cite{SSZ} which rules out $\Gamma_{(0,0)}$ spending a significant amount of time near the boundary of such a cylinder, i.e. in the narrower cylinders $[-w-\delta,-w+\delta] \times [0,1]$ or $[w-\delta,w+\delta] \times [0,1]$).

As already stated, the occupation time of the cylinder $[-w,w] \times [0,1]$ for $\pi_{(0,0)}$ will be denoted by $L_{w}$. The corresponding quantity for $\Gamma_{(0,0)}$ will be denoted by $L^{(n)}_w$ (formal definition appears in \eqref{eq:defn-Ln}).

With the above notation, the preceding discussion implies  that $L^{(n)}_w\to L_{w}$ as $n\to \infty$.
Further, it was shown in \cite[Proposition 2.4]{GZ+} that $L=\lim_{w\to \infty}\frac{L_w}{2w}.$ Thus the proof of Theorem \ref{thm:converge-lc} will be complete once we prove the following key estimate (which most of the paper is devoted to): $$\big|L^{(n)}-\frac{L^{(n)}_w}{2w}\big|=o_w(n^{1/3}),$$ where $o_w(1)$ denotes a quantity that converges to $0$ with $w.$ This will be shown by establishing that the pre-limiting local time exhibits a degree of \emph{smoothness}. More precisely, for any $m,$ we will show that the amounts of time spent by the geodesic $\Gamma_{(0,0)}$ on the lines $\{x-y=m\}$ with $|m|\le wn^{2/3}$, {up to time $n$ (i.e., below the line $\{x+y=2n\}$)}, are very close to each other. For expository purposes, it would be convenient to work with unscaled coordinates in the proof in which the unscaled local time on the line $\{x-y=m\}$ will be denoted by $Z^{(n)}[m]$. In this notation it will suffice to show that $\frac{\E[(Z^{(n)}[m]-Z^{(n)}[0])^2]}{n^{2/3}}=o_{w}(1)$ for all $|m|\le wn^{2/3}$ (recall that $Z^{(n)}[0]$ is expected to be of order $n^{1/3}$ owing to the transversal fluctuations of $\Gamma_{(0,0)}$ being of order $n^{2/3}$ below the line $\{x+y=2n\}$).

Note that it is also natural to speculate that $Z^{(n)}[m]$ varies smoothly as $m$ varies, while the expectation is maximized at $m=0$. Thus we expect $\E[(Z^{(n)}[m]-Z^{(n)}[0])^2]$ to be of order $w^4n^{2/3}$ (due to a second order effect) when $|m|$ is approximately $wn^{2/3}$.
However, we refrain from pursuing such a refined estimate.

We briefly outline in the remainder of this subsection the key steps involved in establishing the above-mentioned smoothness across Sections \ref{ssec:conti-ocp-mea} and \ref{ssec:stab}, which can be considered as the main novelty in the paper. The key input is a stability estimate which shows geodesics, between arbitrary nice enough boundary conditions (formal definition appears in Proposition \ref{prop:stb}), are stable under mild perturbations of the boundary data.  This is done in two steps. First, we show that if the boundary data is perturbed, the original and the perturbed geodesic must be close to each other during a large fraction of their journey owing to certain regularity and decay properties of last passage values as the endpoints are varied. The next step is to show that it is unlikely that two geodesic paths spending a lot of time next to each other do not coalesce. This last statement has already been proven and applied in a few other instances (\cite{GH20, BHS, MSZ}).

Given this general stability, the smoothness of the probability of the geodesic passing through nearby points, and hence the local time, is shown by shifting the noise environment suitably (causing a perturbation of the boundary data as depicted in Figure \ref{fig:42}). Under the natural coupling of the corresponding geodesics, the event of one geodesic passing through a given point and the other geodesic not passing through the corresponding shifted point would entail a non-coalescence of the two geodesics whose probability is shown to be low using the arguments in the preceding paragraph. 

We end this discussion by remarking on a technical point that in fact to perform second-moment computations we would need to show the smoothness of two-point functions, i.e., the probability of the geodesic passing through a pair of points, as one of the two points is varied in space.

\subsection*{Organization of the remaining text}
The rest of this paper is organized as follows.
In Section \ref{sec:prelim} we set up notation conventions, and collect useful results about Exponential LPP and the directed landscape from the literature.
The next three sections are devoted to proving Theorem \ref{thm:converge-lc}.
Section \ref{sec:convlt} gives the general framework by carrying out the second-moment computations and finishes the proof of Theorem \ref{thm:converge-lc}, modulo the smoothness of the LPP local time on different lines (Proposition \ref{prop:Ln-conti}).
In Section \ref{ssec:conti-ocp-mea} we prove Proposition \ref{prop:Ln-conti}, via a key estimate on the two-point distribution of the geodesic $\Gamma_{(0,0)}$ (Proposition \ref{prop:occup-diff}).
In Section \ref{ssec:stab} we prove the stability of geodesics with respect to boundary conditions, which is a crucial input in the proof of Proposition \ref{prop:occup-diff}.
In Appendix \ref{ssec:max-ocp-seg}
we prove a key occupation time estimate (Lemma \ref{lem:disc-max}), using arguments from \cite{SSZ}.

\subsection*{Acknowledgement}
LZ would like to thank Allan Sly for many discussions on the stability of geodesics with respect to boundary conditions.
SG is partially supported by NSF grant DMS-1855688, NSF Career grant DMS-1945172, and a Sloan Fellowship. 
The research of LZ is supported by the Miller Institute for Basic Research in Science, at the University of California, Berkeley.

\section{Preliminaries}  \label{sec:prelim}
\subsection{Notations}
We record some basic notations that will appear throughout the paper (in addition to those already introduced in the previous section). 
For any real numbers $x$ and $y$, let $x\wedge y=\min\{x, y\}$ and $x\vee y=\max\{x, y\}$.
For $x<y,$ we let $\llbracket x,y\rrbracket=[x,y]\cap\Z$.
For any set $A$ we use $|A|$ to denote its cardinality while we will use $\LE$ to denote the Lebesgue measure on $\R$. 
We will also use $\boo$ to denote $(0,0)$. 

We would consider the space of all non-empty compact subsets of $\R^2$.
With the Hausdorff distance, it becomes a metric space, and the induced topology is referred to as the Hausdorff topology.

For the notation $\argmax$ throughout this paper, when using it we do not assume that there is a unique maximum (but there must exist at least one); and when there are multiple ones we take an arbitrary one.

We end by stating a standard practice we will adopt in many of the proofs in this paper where we will use $c,C>0$ to denote small and large constants.
Their specific values will depend on the context and may change from line to line. To avoid repetition, we will not be restating this before every proof.

\subsection{Properties of the directed landscape}

In this subsection, we collect some known results about the directed landscape.

\subsubsection{Basic symmetries and estimates}
We first record some of the various symmetries that the directed landscape satisfies.
\begin{lemma}[\protect{\cite[Lemma 10.2]{DOV}}]  \label{lem:DL-symmetry}
As continuous functions on $\R^4_\uparrow$, $\cL$ is equal in distribution as
\begin{itemize}
    \item (flip invariance) \[
(x,s;y,t) \mapsto \cL(-y,-t;-x,-s);
\]
\item (skew-shift invariance) \[
(x,s;y,t) \mapsto \cL(x+cs+z,s+r;y+ct+z,t+r) + (t-s)^{-1}((x-y-c(t-s))^2 - (x-y)^2)
\]
for any $c,z,r\in \R$ (when $c=0$ this is also referred to as translation invariance);
\item (scaling invariance) \[
(x,s;y,t) \mapsto w\cL(w^{-2}x,w^{-3}s;w^{-2}y,w^{-3}t);
\]
for any $w>0$.
\end{itemize}
\end{lemma}

Another basic property of the directed landscape {(owing to planarity}) is the following quadrangle inequality.
\begin{lemma}[\protect{\cite[Lemma 9.1]{DOV}\cite{BGH}}]  \label{lem:DL-quad}
For any $s<t$, $x_1<x_2$, $y_1<y_2$, we have
\[
\cL(x_1,s;y_1,t) + \cL(x_2,s;y_2,t) \ge \cL(x_1,s;y_2,t) + \cL(x_2,s;y_1,t).
\]
\end{lemma}

We next record some basic estimates on the passage times and its regularities.

\begin{lemma}[\protect{\cite[Corollary 10.7]{DOV}}]  \label{lem:DLbound}
There is a random number $R$ such that the following is true.
First, for any $M>0$ we have $\PP[R>M] < Ce^{-cM^{3/2}}$ for some universal constants $c,C>0$.
Second, for any $u=(x,s;y,t) \in \R_\uparrow^4$, we have
\[
\left| \cL(x,s;y,t) + \frac{(x-y)^2}{t-s}\right| \le R(t-s)^{1/3} \log^{4/3}(2(\|u\|+2)^{3/2}/(t-s)) \log^{2/3}(\|u\|+2). 
\]
\end{lemma}

\begin{lemma}[\protect{\cite[Proposition 1.6]{DOV}}]  \label{lem:modcont}
Let $K$ be a compact subset of $\R_\uparrow^4$.
There is a random number $R$ depending only on $K$, such that the following is true.
First, for any $M>0$ we have $\PP[R>M] < Ce^{-cM^{3/2}}$ for some universal constants $c,C>0$.
Second, for any $(x,s;y,t), (x',s';y',t') \in K$, we have
\begin{multline*}
\left| \left(\cL(x,s;y,t) + \frac{(x-y)^2}{t-s} \right)- \left(\cL(x',s';y',t') + \frac{(x'-y')^2}{t'-s'} \right)\right|
\\ 
\le R\left(\tau^{1/3}\log^{2/3}(1+\tau^{-1}) + \xi^{1/2}\log^{1/2}(1+\xi^{-1}) \right),
\end{multline*}
where $\tau = |s-s'| \vee |t-t'|$ and $\tau = |x-x'| \vee |y-y'|$.
\end{lemma}

\subsubsection{Airy$_2$ process and local Brownian properties}  \label{sssec:atlocbr}
Another central object in the KPZ universality class is the parabolic Airy$_2$ process on $\R$,
denoted as $\cA_2:\R\to\R$.
It had in fact been rigorously established much before 
$\cL$ \cite{prahofer2002PNG}, and has the same 
distribution as the function $\cL(0,0;\cdot,1)$.
On the other hand, by translation invariance of $\cL$ (see Lemma \ref{lem:DL-symmetry}), there is $\cL(0,0;\cdot,1)\overset{d}{=}\cL(y,0; y+\cdot,1)$ for any $y\in \R$.
Therefore the landscape $\cL$ provides a natural coupling of infinitely many parabolic Airy$_2$ processes rooted at different spatial points.

The process $\cA_2$ also has the following Brownian property.
For $K\in \R, d>0$, let $\cB^{[K,K+d]}$ denote a Brownian 
motion on $[K,K+d]$, with $\cB^{[K,K+d]}(K)=0$. Let $\cA_2^{[K,K+d]}$ denote the random function on $[K,K+d]$ defined by 
$$\cA_2^{[K,K+d]}(x):=\cA_2(x)-\cA_2(K), ~\forall x\in [K,K+d].$$

\begin{theorem}[\protect{\cite[Theorem 1.1]{HHJ}}] \label{thm:airytail}
There exists a universal constant $G>0$ such that the following holds.
For any fixed $M>0$, there exists $a_0=a_0(M)$ such that for any interval $[K,K+d]\subset [-M,M]$ and any measurable set $A \subset \cC\big([K,K+d],\mathbb{R} \big)$ (the space of real continuous functions on $[K,K+d]$) with $0<\PP[\cB^{[K,K+d]}\in A]=a\le a_0$, we have
\[\PP\left(\cA_2^{[K,K+d]} \in A\right) \leq a  \exp \Big( G {M} \big( \log a^{-1} \big)^{5/6} \Big).\]
\end{theorem}

We will also need the Brownian absolute continuity of the KPZ fixed point with arbitrary initial conditions.
\begin{theorem}[\protect{\cite[Theorem 1.2]{SV21}}]  \label{thm:abs-bm-gen}
Let $t>0$. Take any function $f_0:\R\to\R\cup\{-\infty\}$, such that $f(x)\neq -\infty$ for some $x$, and $\frac{f_0(x)-x^2/t}{|x|} \to -\infty$ as $|x|\to \infty$. Denote
\[
f_t(y) = \sup_{x\in\R} f_0(x) + \cL(x,0;y,t),
\]
then for any $y_1<y_2$, the law of $y\mapsto f_t(y+y_1)-f(y_1)$ for $y \in [0, y_2-y_1]$ is absolutely continuous with respect to the law of a Brownian motion on $[0, y_2-y_1]$.
\end{theorem}

\subsection{Properties of Exponential LPP}
We next collect some known results on Exponential LPP that will be important for us.

\subsubsection{Busemann functions}  \label{ssec:busemann}
Busemann functions were first used to study geodesics in Riemannian manifolds.
They were introduced  in the context of first passage percolation (FPP) models by Hoffman \cite{Ho05, Ho08}.
We now set up the Busemann function in Exponential LPP, which has also been intensively studied and widely used (see e.g. \cite{Se17}). 

By the tree structure of semi-infinite geodesics (already alluded to in the earlier discussion on Exponential LPP), for each $p\in\Z^2$ we let $G(p)= T_{p,\bc} - T_{\boo,\bc}$, where $\bc\in\Z^2$ is the coalescing point of $\Gamma_p$ and $\Gamma_\boo$;
i.e., $\bc$ is the vertex in $\Gamma_p \cap \Gamma_\boo$ with the smallest $d(\bc)$.
Such Busemann function $G$ satisfies the following properties.
\begin{enumerate}
\item For any $p=(x,y)\in\Z^2$ and $n>y$, if we take $m_*=\argmax_{m\ge x} T_{p,(m,n)}+G(m,n+1)$, then $\Gamma_p = \Gamma_{p,(m_*,n)}\cup \Gamma_{(m_*,n+1)}$, and $G(p)=T_{p,(m_*,n)}+G((m_*,n+1))$.
\item For any down-right path $\cU=\{u_k\}_{k\in\Z}$, the random variables $G(u_k)-G(u_{k-1})$ are independent. The law of $G(u_k)-G(u_{k-1})$ is $\Exp(1/2)$ if $u_k=u_{k-1}-(0,1)$, and is $-\Exp(1/2)$ if $u_k=u_{k-1}+(1,0)$. They are also independent of $\omega(p)$ for all $p\in \cU_-$, where $\cU_-$ is the lower part of $\Z^2\setminus \cU$,
\end{enumerate}
The first property is a consequence of the definition of $G$. For the second property, its proof can be found in \cite{Sep20}.

\subsubsection{Passage time estimates}  \label{sssec:ptest}

The passage time $T_{\boo,(m,n)}$ has the same law as the largest eigenvalue of $X^*X$, where $X$ is an $(m+1)\times (n+1)$ matrix of i.i.d.\ standard complex Gaussian entries (see \cite[Proposition 1.4]{Jo20}).
Hence we get the following one-point estimates from \cite[Theorem 2]{LR10}.
\begin{theorem}
\label{thm:lpp-onepoint}
There exist constants $C,c>0$, such that for any $m\geq n \geq 1$ and $x>0$, we have
\begin{equation}  \label{e:wslope}
\PP[T_{\boo, (m,n)}-(\sqrt{m}+\sqrt{n})^{2} \geq xm^{1/2}n^{-1/6}] \leq Ce^{-cx}.    
\end{equation}
In addition, for each $\psi>1$, there exist $C',c'>0$ depending on $\psi$ such that if $\frac{m}{n}< \psi$, we have
\begin{equation}  \label{e:slope}
\begin{split}
&\PP[T_{\boo, (m,n)}-(\sqrt{m}+\sqrt{n})^{2} \geq xn^{1/3}] \leq C'e^{-c'\min\{x^{3/2},xn^{1/3}\}},\\
&\PP[T_{\boo, (m,n)}-(\sqrt{m}+\sqrt{n})^{2} \leq -xn^{1/3}] \leq C'e^{-c'x^3},
\end{split}
\end{equation}
and as a consequence
\begin{equation}
\label{e:mean}
|\E T_{\boo, (m,n)} -(\sqrt{m}+\sqrt{n})^2|\leq C'n^{1/3}.
\end{equation}
\end{theorem}
We also have the following parallelogram estimate.
\begin{prop}  [\protect{\cite[Theorem 4.2]{BGZ}}]
\label{prop:seg-to-seg}
Let $U$ be the parallelogram whose one pair of sides have length $2n^{2/3}$ and are aligned with $\LL_0$ and $\LL_{n}$ respectively, with their midpoints being $(m,-m)$ and $\bn$.
Let $U_1, U_2$ be the parts of $U$ below $\LL_{n/3}$ and above $\LL_{2n/3}$ respectively.
For each $\psi<1$, there exist constants $C,c>0$ depending only on $\psi$, such that when $|m|<\psi n$,
\[
\PP\Big[\sup_{u\in U_1, v\in U_2} |T_{u,v}-\E T_{u,v}| \geq x n^{1/3}\Big]\leq Ce^{-cx}.
\]
\end{prop}
Such a result was first proved as \cite[Proposition 10.1, 10.5]{BSS}, in the setting of Poissionian LPP.
For Exponential LPP, a proof is given in \cite[Appendix C]{BGZ}, following similar ideas.

\subsubsection{Transversal fluctuation}   \label{sssec:transf}
We will rely on the following transversal fluctuation estimates of geodesics, which can be proved using Proposition \ref{prop:seg-to-seg} and some geometric arguments.
Such estimates have appeared and been used in several works \cite{BGZ,MSZ,BSS,BSS19,Z20}.
We quote the following versions for point-to-point and semi-infinite geodesics.
\begin{lemma}[\protect{\cite[Proposition C.9]{BGZ}}] \label{lem:trans-dis-fini}
For any $0<\psi<1$, there exist constants $C,c>0$ such that the following is true.
Take any $n\in\N$ large enough and any $m\in \Z, |m|<\psi n$, $x>0$.
Consider the parallelogram whose four vertices are
$(-xn^{2/3}, xn^{2/3})$, $(xn^{2/3}, -xn^{2/3})$, and $(n+m-xn^{2/3}, n-m-xn^{2/3})$, $(n+m+xn^{2/3}, n-m+xn^{2/3})$.
Then with probability $1-Ce^{-cx^3}$, the geodesic $\Gamma_{\boo,(n+m,n-m)}$ is contained in that parallelogram.
\end{lemma}

\begin{lemma}[\protect{\cite[Corollary 2.11]{MSZ}}] \label{lem:trans-dis-inf}
There exist constants $C,c>0$ such that the following is true.
Take any $n\in\N$ large enough and any $x>0$, and consider the rectangle whose four vertices are
$(-xn^{2/3}, xn^{2/3})$, $(xn^{2/3}, -xn^{2/3})$, and $(n-xn^{2/3}, n+xn^{2/3})$, $(n+xn^{2/3}, n-xn^{2/3})$.
Then with probability $1-Ce^{-cx^3}$, the part of the geodesic $\Gamma_\boo$ below $\LL_n$ is contained in that rectangle.
\end{lemma}

\subsubsection{Intersections of geodesics}
\label{ssec:mdg}

The following estimate from \cite{BHS} bounds the number of intersections that geodesics can have with a line.

\begin{prop} [\protect{\cite[Proposition 3.10]{BHS}}]  \label{prop:num-int-dis-weak}
Let $A_n$ and $B_n$ be line segment along the lines $\LL_0$ and $\LL_n$, with midpoints $(xn^{2/3}, -xn^{2/3})$ and $(n,n)$, respectively, and each has length $2n^{2/3}$.
Let $F_{n,m}=\left(\bigcup_{u\in A_n, v\in B_n} \Gamma_{u,v}\right) \cap \LL_m$, for $0<m<n$.
For any $\psi \in (0,1)$, $\alpha\in (0, 1/2)$,  there is a constant $c>0$, such that when $|x|<\psi n^{1/3}$ and $\alpha n < m < (1-\alpha)n$, we have
\[
\PP[|F_{n,m}| > M] < e^{-cM^{1/128}}
\]
for any sufficiently large $M<n^{0.01}$ and sufficiently large $n$.
\end{prop}
The above statement is a consequence of geodesic coalescence which causes only a few (order $1$ many) highways passing through the line $\LL_m$ in an order $n^{2/3}$ window and all the geodesics ($\Gamma_{u,v}$ for some $u\in A_n$, $v\in B_n$) merging with one of them. 

We note that Proposition \ref{prop:num-int-dis-weak} is slightly more general than \cite[Proposition 3.10]{BHS}, which is only for $\alpha = 1/3$; but the proof in \cite{BHS} goes through verbatim for general $\alpha \in (0,1)$.

In Proposition \ref{prop:num-int-dis-weak}, we considered geodesics with endpoints varying on some segments of length in the order of $n^{2/3}$, and their intersections with the line $\LL_m$.
We next state a different version that allows the upper end of the geodesics to vary on the whole line $\LL_n$, but we only consider their intersections with a segment of the length of the order of $n^{2/3}$.
\begin{lemma}  \label{lem:num-int-dis}
Take any $m<n\in\N$ and $\ell>1$, let $F=\left(\bigcup_{|i|<\ell n^{2/3}, j \in \Z} \Gamma_{(i,-i),(n+j,n-j)}\right) \cap \{(m+i,m-i): |i|<\ell n^{2/3}\} |$.
Then for each $0<\alpha < 1/2$, there exist $c,C>0$, such that $\PP[|F| > M] < Ce^{-c M^{1/12800}}$ when $\ell < cn^{1/3}\wedge M^{1/3}$, and $\alpha n < m < (1-\alpha)n$.
\end{lemma}
Transversal fluctuation estimates (from Section \ref{sssec:transf}) imply that for any geodesic from some $(i,-i)$ with $|i|<\ell n^{2/3}$ to $\LL_n$, if it intersects $\{(m+i,m-i): |i|<\ell n^{2/3}\}$, the distance between its upper end and $(n,n)$ is likely to be of order (at most) $n^{2/3}$.
Therefore we can apply Proposition \ref{prop:num-int-dis-weak}.
\begin{proof}[Proof of Lemma \ref{lem:num-int-dis}]
We let $c,C>0$ denote small and large constants depending on $\alpha$.
We may also assume that $M< 2\ell n^{2/3}+1<n$,
since otherwise we would always have $|F| \le M$.

Let $\ell' = C\ell+M^{1/10}<n^{1/3}/2$. 
Consider $F'=\left(\bigcup_{|i|<\ell n^{2/3}, |j|<\ell' n^{2/3}} \Gamma_{(i,-i),(n+j,n-j)}\right) \cap \{(m+i,m-i): |i|<\ell n^{2/3}\} |$.
Then under $|F|>M$, we have either $|F'|>M$, or there are some $|i|<\ell n^{2/3}$ and $|j|\ge \ell' n^{2/3}$, such that $\Gamma_{(i,-i),(n+j,n-j)}$ intersects $\LL_m$ at some $(m+k,m-k)$, for $|k|<\ell n^{2/3}$.
This would imply that either $\Gamma_{(\lfloor \ell n^{2/3} \rfloor, -\lfloor \ell n^{2/3} \rfloor), (n-\lfloor \ell' n^{2/3} \rfloor, n+\lfloor \ell' n^{2/3} \rfloor)}$ intersects $\LL_m$ at some $(m+k,m-k)$, for $k>-\ell n^{2/3}$;
or $\Gamma_{(-\lfloor \ell n^{2/3} \rfloor, \lfloor \ell n^{2/3} \rfloor), (n+\lfloor \ell' n^{2/3} \rfloor, n-\lfloor \ell' n^{2/3} \rfloor)}$ intersects $\LL_m$ at some $(m+k,m-k)$, for $k<\ell n^{2/3}$.
By Lemma \ref{lem:trans-dis-fini}, these happen with probability $<Ce^{-cM^{3/10}}$.

Finally, by splitting $\{(i,-i):|i|<\ell n^{2/3}\}$ and $\{(n+j,n-j):|j|<\ell' n^{2/3}\}$ into segments of length $n^{2/3}$, we can write $F'$ as the union of $\ell\ell'$ sets; and for each set, the probability that its size is $>M^{0.01}$ is at most $e^{-c(M^{0.01})^{1/128}}$, by Proposition \ref{prop:num-int-dis-weak} (note that $M^{0.01}<n^{0.01}$ by our assumption above).
Then since $\ell \ell' M^{0.01}<M$, we have
\[
\PP[|F'|>M] \le \PP[|F'|>\ell \ell' M^{0.01}]< \ell \ell' e^{-c(M^{0.01})^{1/128}}<Ce^{-cM^{1/12800}}.
\]
Thus the conclusion follows.
\end{proof}

We can get similar results for semi-infinite geodesics (in the $(1,1)$-direction). In this case, note that the lower endpoint can vary among all points in $\LL_0$.
\begin{lemma}  \label{lem:num-int-dis-inf}
For any $n\in\N$ and $\ell > 1$, let $F=(\bigcup_{i\in \Z} \Gamma_{(i,-i)}) \cap \{(n+i,n-i): |i|<\ell n^{2/3}\} $.
Then there exist $c,C>0$, such that $\PP[|F| > M] < Ce^{-c M^{1/12800}}$ when $\ell < cn^{1/3}\wedge M^{1/3}$.
\end{lemma}

\begin{proof}
In this proof we let $c,C>0$ denote small and large constants.
We can also assume that $M< 2\ell n^{2/3}+1<n$,
since otherwise we would always have $F\le M$.

Let $\ell' = C\ell+M^{1/10}<n^{1/3}/2$.
Consider $F'=\left(\bigcup_{|i|, |j|<\ell' n^{2/3}} \Gamma_{(i,-i),(2n+j,2n-j)}\right) \cap \{(n+i,n-i): |i|<\ell n^{2/3}\} |$.
Then under $|F|>M$, we have either $|F'|>M$, 
or there are some $|i|\ge \ell' n^{2/3}$, such that $\Gamma_{(i,-i)}$ intersects $\LL_n$ at some $(n+k,n-k)$, for $|k|<\ell n^{2/3}$.
This would imply that either $\Gamma_{(\lfloor \ell' n^{2/3} \rfloor, -\lfloor \ell' n^{2/3} \rfloor)}$ intersects $\LL_n$ at some $(n+k,n-k)$, for $k>-\ell n^{2/3}$;
or $\Gamma_{(-\lfloor \ell' n^{2/3} \rfloor, \lfloor \ell' n^{2/3} \rfloor)}$ intersects $\LL_n$ at some $(n+k,n-k)$, for $k<\ell n^{2/3}$.
By Lemma \ref{lem:trans-dis-inf}, we would have that these happen with probability $<Ce^{-cM^{3/10}}$.

Finally, as in the proof of Lemma \ref{lem:num-int-dis}, we split $\{(i,-i):|i|<\ell' n^{2/3}\}$ and $\{(2n+j,2n-j):|j|<\ell' n^{2/3}\}$ into segments of length $n^{2/3}$, and apply  Proposition \ref{prop:num-int-dis-weak} to each pair.
We then conclude that $\PP[|F'|>M]<Ce^{-cM^{1/12800}}$, so the conclusion follows.
\end{proof}

\subsection{Convergence of Exponential LPP to the directed landscape}   \label{ssec:lpptodl}
In \cite{DV21} it is proved that Exponential LPP converges to the directed landscape after suitable scaling. We record the formal statement in this section.

To state the convergence, we need a linear transformation of Exponential LPP to the scaling of the directed landscape. Among other things, the transformation redirects the orientation from the $(1,1)$-direction to the vertical direction.

For any $p=(x,y)\in\R^2$ we let $\fG_n (p) = ( 2^{-5/3}n^{-2/3}(x-y), n^{-1}y)$.
For any set $A\subset \R^2$ and $n\in \N$, we let $\fG_n(A) = \{( \fG_n(p) : p \in A\}$.
We then define a function $\cK_n$ on $\R^4$, the transformed last-passage times.
It is also a ``directed metric'' on $\R^2$ (as it satisfies a composition law similar to \eqref{eq:DL-compo}).
For any $(x,s;y,t) \in \R^4$, we let
\[
\cK_n(x,s;y,t) = 2^{-4/3}n^{-1/3} F(\fG_n^{-1}(x,s); \fG_n^{-1}(y,t)).
\]
where $F:\R^4 \to \R$ is the rounded version of last-passage times, defined as follows.
Let $\fr:\R^2 \to \Z^2$ be the ``rounding'' function where
\[
\fr(x,y) =
\begin{cases}
(x, y)  &  x, y \in \Z;\\
(\lfloor x \rfloor, y) &  x\not\in \Z, y \in \Z;\\
(x, \lfloor y \rfloor) &  x\in \Z, y \not\in \Z;\\
(\lceil x\rceil, \lfloor y \rfloor) &  x\not\in \Z, y \not\in \Z.
\end{cases}
\]
We then let $F(x,y;z,w)= T_{(\lceil x\rceil, \lceil y\rceil), \fr(z,w)} - 2d(\fr(z,w)-\fr(x,y)) - \omega_{(x,y)}\don[x, y\in \Z]$.
Here $2d(\fr(z,w)-\fr(x,y))$ is the leading linear approximation of $T_{(\lceil x\rceil, \lceil y\rceil), \fr(z,w)}$, and $\omega_{(x,y)}\don[x, y\in \Z]$ is the correction term to make $F$ satisfy a composition law (similar to \eqref{eq:DL-compo}).
We note that $\cK_n(x,s;y,t)=-\infty$ if $(\lceil ns+2^{5/3}n^{2/3}x\rceil, \lceil ns\rceil) \not\preceq \fr(nt+2^{5/3}n^{2/3}y, nt)$, as we let $T_{p,q}=-\infty$ when  $p\not\preceq q$.

For any $p, q\in \R^2$, a set $A\subset \R^2$ is called a \emph{geodesic set} in $\cK_n$ from $p$ to $q$, if there is a total ordering $\preceq_A$ on $A$, such that $p\preceq_A p' \preceq_A q$ for any $p'\in A$; 
and for any $p_1\preceq_A p_2 \preceq_A p_3 \in A$, 
\begin{equation}  \label{eq:cKcomp}
\cK_n(p_1;p_2) + \cK_n(p_2;p_3) = \cK_n(p_1;p_3).
\end{equation}
Such a set $A$ is called a \emph{maximal geodesic set} in $\cK_n$ if it is not contained in any other geodesic set from $p$ to $q$.
From these definitions, we have that for any $p, q\in\Z^2$, the set $\fG_n(\Gamma_{p,q})$ is a maximal geodesic set from $\fG_n(p)$ to $\fG_n(q)$.
Indeed, obviously $\fG_n(\Gamma_{p,q})$ is a geodesic set, and it is maximal because adding any extra point (with any ordering) would violate \eqref{eq:cKcomp} (while this can be checked straightforwardly, we omit the somewhat tedious details.)

Finally, for any continuous path $\pi:I\to\R$, where $I\subset \R$ is any subset, we define its \emph{graph} $\Graph(\pi)$ as the set $\{(\pi(t),t): t\in I\}$. We then have the following convergence result. 
\begin{theorem}[\protect{\cite[Theorem 13.8(2)]{DV21}}]  \label{thm:exp-to-dl}
There is a coupling of $\cL$ with $\cK_n$ for all $n\in\N$, such that the following holds almost surely.
First, for any compact $K\subset \R^4_\uparrow$, we have $\cK_n \to \cL$ uniformly in $K$.
Second, for any $(p_n; q_n) \to (p; q) \in \R^4_\uparrow$,
and any $\cK_n$ maximal geodesic set from $p_n$ to $q_n$, denoted as $\pi_n$, they are pre-compact in Hausdorff topology, and any subsequential limit is $\Graph(\pi)$, for some $\cL$ geodesic $\pi$ from $p$ to $q$.
\end{theorem}

The convergence of the Exponential LPP semi-infinite geodesic $\Gamma_\boo$ to $\pi_\boo$ is also established, in the Hausdorff topology. 
\begin{prop}[\protect{\cite[Proposition 5.1]{GZ+}}]  \label{prop:cov-semi-geo}
For any $h>0$, $\fG_n(\Gamma_\boo) \cap \R\times [0,h] \to \Graph(\pi_\boo) \cap \R\times [0,h]$ in distribution in the Hausdorff topology, as $n\to\infty$.
\end{prop}
This is also obtained in \cite[Lemma 4.2]{RV21}.

\subsection{Occupation estimates}

We will rely on the following crucial tail estimate established in \cite{GZ+} on the occupation time of a semi-infinite geodesic in the directed landscape based on an earlier similar estimate for Exponential LPP obtained in \cite{SSZ}.
\begin{lemma}[\protect{\cite[Lemma 2.7]{GZ+}}]  \label{lem:geo-ocp-tail}
There are universal constants $C,c>0$ such that the following is true.
For any closed interval $I\subset \R$ and $J\subset \R_{\ge 0}$, we have
\[
\PP\left[\int_J \don[\pi_\boo(t)\in I] dt > M\LE(I)\LE(J)^{1/3}\right] < Ce^{-cM}
\]
for any $M>0$.
\end{lemma}

One obtains the following continuity result as a straightforward consequence of the above. 
\begin{lemma}\label{lem:L-conti-w}
Almost surely the following holds.
For any $h\ge 0$, $w\mapsto L_w(h)$ is continuous for $w\ge 0$.
\end{lemma}
\begin{proof}
By Lemma \ref{lem:geo-ocp-tail}, we have that there exists a random $R>0$, such that for any $i,a,b\in \Z$, we have
\[
\int_0^{2^a} \don[\pi_\boo(t)\in [i2^b, (i+2)2^b] ] dt < R(|a|+|b|+|i|^{1/2}+1) 2^{b+a/3}. 
\]
Indeed, for any fixed $i, a, b\in \Z$ and $R>0$, the probability for this inequality to fail is exponentially small in $R(|a|+|b|+|i|^{1/2}+1)$. Thus it is summable when taking a union bound over $i, a, b$.

For any $h>0$, $w>0$, and $|\delta|<w$, by taking $a=\lceil \log_2(h)\rceil$, $b=\lceil \log_2(\delta)\rceil$, and making appropriate choices of $i$, we have that
\[
|L_w(h)-L_{w+\delta}(h)| < 16R \delta h^{1/3} (|\log_2(h)| + |\log_2(\delta)| + (|w|/|\delta|)^{1/2} +4).
\]
So the conclusion follows.
\end{proof}

We will also use a version of the occupation time bound in the Exponential LPP setting.
For any $a<b \in \Z$, we let $W[a,b]=|\{a\le k \le b: (k,k)\in \Gamma_{(a,a),(b,b)}\}|$, and let $W_{a,b}=\max_{a\le a' \le b' \le b}W[a', b']$.

\begin{lemma}  \label{lem:disc-max}
There are universal constants $C,c>0$ such that the following is true. Take any large enough $A$.
For any $a<b \in \Z$, there is
\begin{equation}  \label{eq:dis-W-exp}
\E[W_{a,b}] < A^{2/3+0.01} (b-a)^{1/3},    
\end{equation}
and for any $M>0$,
\begin{equation}  \label{eq:dis-W-tail}
\PP[W_{a,b} > AM(b-a)^{1/3}] < Ce^{-cM}.
\end{equation}
\end{lemma}
This result in a slightly different setting has appeared in \cite{SSZ}. 
For completeness, we give a proof of Lemma \ref{lem:disc-max} in the appendix, following a similar strategy.

\section{Convergence to the local time}
\label{sec:convlt}

We start by recalling briefly our strategy to prove Theorem \ref{thm:converge-lc}.
The key input is Proposition \ref{prop:cov-semi-geo}, which is the convergence of the semi-infinite geodesic in Exponential LPP to the one in the directed landscape, in the Hausdorff topology. However, the scaling of Proposition \ref{prop:cov-semi-geo} is too coarse to get information on a line. Thus, to use the convergence result, we need to relate the time spent in a strip with the local time on a line. This will be shown by establishing certain smoothness properties of the local time spent on a line as the latter moves in the spatial direction.
Towards this, we consider the time spent in a strip, defined as follows:
for any $w>0$ and $n\in\N$, we define
\begin{equation}  \label{eq:defn-Ln}
L^{(n)}_w(h)=(2n)^{-1}|\{(x,y): 0\le x+y \le 2hn, |x-y| \le 2^{5/3}n^{2/3}w, (x,y) \in \Gamma_\boo \}|.    
\end{equation}
This is a discrete analog of the function $L_w$, and the scalings match those in Section \ref{ssec:lpptodl} to make it converge to $L_w$ as $n\to\infty$.

The proof of Theorem \ref{thm:converge-lc} consists of the following components.
\begin{enumerate}
    \item $(2w)^{-1}L_w(h) \to L(h)$ as $w\to 0$ {almost surely}.
    \item For any $\delta>0$, and any $w>0$ is small enough (depending on $\delta$), we have
    $|(2w)^{-1}L^{(n)}_w(h) - L^{(n)}(h)|<\delta$ for any $n$ large enough (depending on $\delta, w$).
    \item For any $w>0$, $L^{(n)}_w(h)\to L_w(h)$ as $n\to \infty$ (in a sense that will be clear later).
\end{enumerate}
Given the above statements, by first taking $w$ small enough, then $n$ large enough, it follows that $L^{(n)}\to L$ in probability.

The first step has been accomplished in \cite{GZ+} and is recorded as Proposition \ref{prop:om-limit}.
The third step will involve applying Proposition \ref{prop:cov-semi-geo}, and is stated as Lemma \ref{lem:L-coupl-conv} below. The main task that remains is to carry out the second step, which is formulated as the following result.
\begin{prop}  \label{prop:Ln-conti}
For each $h \ge 0$, we have $\lim_{w\to 0} \limsup_{n\to\infty}\E[((2w)^{-1}L^{(n)}_w(h)-L^{(n)}(h))^2] =0$.
\end{prop}
We leave the proof of this proposition for later and perform the third step now.
The argument will not depend on the exact value of $h$ and hence for concreteness, we shall assume $h=1$.

By Proposition \ref{prop:cov-semi-geo} and using Skorokhod's representation theorem (see e.g. \cite[Chapter 1, Section 6]{billingsley2013convergence}, for the space of all compact subsets of $\R^2$ with the Hausdorff distance metric, which is separable), we couple a sequence of $\Z^2$ Exponential LPP models and the directed landscape, such that if we denote by $\Gamma_\boo^{(n)}$ the $(1,1)$-direction semi-infinite geodesic from $\boo$ in the $n$-th Exponential LPP, we have $\fG_n(\Gamma_\boo^{(n)}) \cap \R\times [0,2] \to \Graph(\pi_\boo) \cap \R\times [0,2]$ in the Hausdorff metric.
We then let $L^{(n)}_w$ be defined using $\Gamma_\boo^{(n)}$.
Then the following convergence holds.
\begin{lemma}  \label{lem:L-coupl-conv}
For any $w>0$, almost surely we have $L^{(n)}_w \to L_w$ uniformly in any compact set, as $n\to\infty$.
\end{lemma}

To prove this lemma, it suffices to use the convergence given by Proposition \ref{prop:cov-semi-geo}, and bound the time that the geodesic {$\pi_\boo$} spends in $[-w-\delta, -w+\delta]\cup [w-\delta, w+\delta]$, as illustrated in Figure \ref{fig:32}.
For the latter task, we use Lemma \ref{lem:L-conti-w}.

\begin{figure}[t]
         \centering
\begin{tikzpicture}
[line cap=round,line join=round,>=triangle 45,x=0.9cm,y=0.2cm]
\clip(1,-2) rectangle (9,22);

\fill[line width=0.pt,color=blue,fill=blue,fill opacity=0.15]
(6.4,0) -- (6.4,100) -- (6.2,100) -- (6.2,0) -- cycle;
\fill[line width=0.pt,color=green,fill=green,fill opacity=0.15]
(3.8,0) -- (3.8,100) -- (6.2,100) -- (6.2,0) -- cycle;

\fill[line width=0.pt,color=blue,fill=blue,fill opacity=0.15]
(3.8,0) -- (3.8,100) -- (3.6,100) -- (3.6,0) -- cycle;

\draw (-1,0) -- (11,0);

\draw [dashed] (-1,18) -- (11,18);

\draw [fill=uuuuuu] (5,0) circle (1.5pt);

\draw [line width=0.6pt] (5,0) -- (4.2,2) -- (5.1,4) -- (5.8,6) -- (5.3,8) -- (6.9,10) -- (7.1,12) -- (6.22,14) -- (6.35,16) -- (5.7,18) -- (6.3,20) -- (6.6,22);

\begin{scriptsize}
\draw (5,0) node[anchor=north]{$\boo$};
\draw (4,18) node[anchor=south]{$t=1$};
\draw (5.3,17) node[anchor=west]{$\pi_\boo$};

\draw [darkgreen] (2,16.6) node[anchor=south west]{$J_\delta^{in}$};
\draw [darkgreen] (2,4) node[anchor=south west]{$J_\delta^{in}$};

\draw [blue] (7.5,9.1) node[anchor=east]{$J_\delta$};
\draw [blue] (7.5,20.2) node[anchor=east]{$J_\delta$};
\draw [blue] (7.5,15) node[anchor=east]{$J_\delta$};

\draw [red] (8.5,21) node[anchor=east]{$J_\delta^{out}$};
\draw [red] (8.5,11) node[anchor=east]{$J_\delta^{out}$};

\end{scriptsize}

\draw [blue] [thick] [|-|] (7.5,13.25) -- (7.5,16.8);
\draw [blue] [dotted] (6.2,13.25) -- (7.5,13.25);
\draw [blue] [dotted] (6.2,16.8) -- (7.5,16.8);

\draw [blue] [thick] [|-|] (7.5,19.4) -- (7.5,20.97);
\draw [blue] [dotted] (6.2,19.4) -- (7.5,19.4);
\draw [blue] [dotted] (6.2,20.97) -- (7.5,20.97);

\draw [blue] [thick] [|-|] (7.5,8.85) -- (7.5,9.6);
\draw [blue] [dotted] (6.2,9.6) -- (7.5,9.6);
\draw [blue] [dotted] (6.2,8.85) -- (7.5,8.85);

\draw [darkgreen] [thick] [|-|] (2,0) -- (2,9.35);
\draw [darkgreen] [dotted] (6.2,9.35) -- (2,9.35);

\draw [darkgreen] [thick] [|-|] (2,16.3) -- (2,19.9);
\draw [darkgreen] [dotted] (6.2,19.9) -- (2,19.9);
\draw [darkgreen] [dotted] (6.2,16.3) -- (2,16.3);

\draw [red] [thick] [|-|] (8.5,20.5) -- (8.5,30);
\draw [red] [dotted] (6.4,20.5) -- (8.5,20.5);

\draw [red] [thick] [|-|] (8.5,9.2) -- (8.5,13.8);
\draw [red] [dotted] (6.4,13.8) -- (8.5,13.8);
\draw [red] [dotted] (6.4,9.2) -- (8.5,9.2);

\end{tikzpicture}
         \caption{An illustration of the proof of Lemma \ref{lem:L-coupl-conv}: to bound the time $\pi_\boo$ spends near the left and right boundaries of $[-w,w]\times [0,1]$.
         The green, blue, and red segments illustrate the time intervals of $J_\delta^{in}$, $J_\delta$, $J_\delta^{out}$, respectively.}
         \label{fig:32}
     \end{figure}
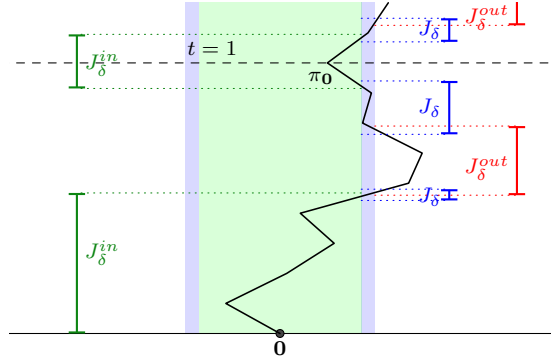

\begin{proof}[Proof of Lemma \ref{lem:L-coupl-conv}]
Without loss of generality we show that $L^{(n)}_w \to L_w$ uniformly in $[0, 1]$.
Given the coupling and the convergence in the Hausdorff topology, the main thing to consider is the set $\Graph(\pi_\boo)$ near the boundary of $[-w, w]\times [0,1]$.
More precisely, for the limit of $L^{(n)}_w$ as $n\to\infty$, we will upper bound it by $L_{w+\delta}$ and lower bound it by $L_{w-\delta}$, for any small $\delta>0$. Then the conclusion follows by the continuity of $L_w$ in $w$.

To connect $L^{(n)}_w$ to $L_w$, we start by re-writing
\[
L^{(n)}_w(h)=(2n)^{-1}|\{(x,s) \in \fG_n(\Gamma_\boo^{(n)}) : 0\le 2^{5/3}n^{2/3}x+2ns \le 2hn, |x| \le w \}|,
\]
for any $h\ge 0$.
This directly follows from \eqref{eq:defn-Ln} and the fact that $\fG_n^{-1}(x,s)=(2^{5/3}n^{2/3}x+ns, ns)$.
We note that for any two different $(x,s), (y,r)\in \fG_n(\Gamma_\boo^{(n)})$, the numbers $2^{5/3}n^{2/3}x+2ns=d(\fG_n^{-1}(x,s))$ and $2^{5/3}n^{2/3}y+2nr=d(\fG_n^{-1}(y,r))$ are different non-negative integers.

We now define
\[
J^{(n)} = \{0\le t \le 1: \exists (x, s) \in \fG_n(\Gamma_\boo^{(n)}), |2nt-(2^{5/3}n^{2/3}x+2ns)| \le 1/2, |x| \le w \}.
\]
Then it equals
\[
[0,1]\cap \bigcup_{(x,s)\in \fG_n(\Gamma_\boo^{(n)}), |x|\le w} [2^{2/3}n^{-1/3}x+s-(4n)^{-1}, 2^{2/3}n^{-1/3}x+s+(4n)^{-1}].
\]
Recall that $\LE$ denotes the Lebesgue measure on $\R$.
Also note that for all $(x,s)\in\fG_n(\Gamma_\boo^{(n)})$, $2^{5/3}n^{2/3}x+2ns$ are mutually different integers.
Then 
\[|\{(x,s)\in\fG_n(\Gamma_\boo^{(n)}): 0\le 2^{2/3}n^{-1/3}x+s \le h, |x|\le w\}| \in \{\lceil 2n\LE(J^{(n)}\cap [0,h]) \rceil, \lfloor 2n\LE(J^{(n)}\cap [0,h]) \rfloor\}.
\]
Therefore $|L^{(n)}_w(h) - \LE(J^{(n)}\cap [0,h])| \le (2n)^{-1}$ for $h\in [0,1]$.
It then suffices to prove the convergence of $\LE(J^{(n)}\cap [0,\cdot])$ to $L_w$ (uniformly in $[0,1]$). For this, we consider the following quantities.
Take any $0< \delta < w$.
We denote
\begin{align*}
    J_\delta &= \{t\ge 0: \exists t', |t'-t|\le \delta, w-\delta \le |\pi_\boo(t')| \le w+\delta\} \\
    J_\delta^{out} &= \{t\ge 0: \exists t', |t'-t|\le \delta, |\pi_\boo(t')| \ge w+\delta\} \\
    J_\delta^{in} &= \{t\ge 0: \exists t', |t'-t|\le \delta, |\pi_\boo(t')| \le w-\delta\}.
\end{align*}
See Figure \ref{fig:32} for an illustration of these sets.
We next show that for any $\delta>0$,
\begin{equation} \label{eq:J-includ-rela}
J^{(n)}\subset J_\delta \cup J_\delta^{in}, \quad [0,1]\setminus J^{(n)} \subset J_\delta \cup J_\delta^{out},
\end{equation}
for $n$ large enough. 
For any $t\in J^{(n)}$, we take $(x, s) \in \fG_n(\Gamma_\boo^{(n)})$ such that $|x| \le w$ and $|2nt-(2^{5/3}n^{2/3}x+2ns)| \le 1/2$.
When $n$ is large enough, we have $|2^{2/3}n^{-1/3}x|<\delta/3$, therefore $|s-t|<\delta/3+(4n)^{-1}$;
and we can find $(x',s') \in \Graph(\pi_\boo)$ such that $|x-x'|, |s-s'| < \delta/3$ (by almost sure convergence under the coupling).
Thus \[|s'-t|\le |s-s'| + |s-t|<2\delta/3+(4n)^{-1}\] and \[|\pi_\boo(s')|=|x'|\le |x|+|x-x'| < w+\delta/3.\] So we have $t \in J_\delta \cup J_\delta^{in}$, and we conclude that $J^{(n)}\subset J_\delta \cup J_\delta^{in}$. By a similar argument we have that $[0,1]\setminus J^{(n)} \subset J_\delta \cup J_\delta^{out}$.

By \eqref{eq:J-includ-rela}, when $n$ is large enough, we have that 
\[
\LE(J^{(n)}\cap [0,h]) \le \LE((J_\delta \cup J_\delta^{in}) \cap [0,h])  = L_{w+\delta}(h),
\]
and
\[
\LE(J^{(n)}\cap [0,h]) \ge h-\LE((J_\delta \cup J_\delta^{out}) \cap [0,h]) = \int_0^h \don[\pi_\boo(t) \in (-w+\delta, w-\delta)] dt  {\ge} L_{w-2\delta}(h),
\]
for any $h\in [0,1]$.
Note that $0\le L_{w+\delta}(h)-L_w(h)\le L_{w+\delta}(1)-L_w(1)$ and $0\ge L_{w-2\delta}(h)-L_w(h)\ge L_{w-2\delta}(1)-L_w(1)$ for any $h\in [0,1]$.
Thus by sending $\delta \to 0$, and using Lemma \ref{lem:L-conti-w}, we conclude that almost surely, $\lim_{n\to\infty} \LE(J^{(n)}\cap [0,h]) = L_w(h)$ uniformly for $h\in [0, 1]$, and the conclusion follows.
\end{proof}

We now assemble the three steps to prove Theorem \ref{thm:converge-lc}.
\begin{proof}[Proof of Theorem \ref{thm:converge-lc}]
Without loss of generality, we prove that $L^{(n)}\to L$ as random functions on $[0, 1]$, weakly in the topology of uniform convergence.

Take the above coupling between a sequence of Exponential LPP models and the directed landscape, and consider the induced coupling between $\Gamma_\boo^{(n)}$, $L^{(n)}$, $L^{(n)}_w$ and $\pi_\boo$, $L$, $L_w$, for all $n\in\N$ and $w>0$.
It is worth remarking that under this coupling, the convergence $L^{(n)}\to L$ can be proved in even stronger senses than the theorem states.

By Lemma \ref{lem:L-coupl-conv}, we assume the almost sure event that for each $m\in \N$, $\lim_{n\to\infty}L^{(n)}_{1/m} = L_{1/m}$ uniformly as functions on $[0,1]$.

For any $m\in \N$ and $0\le h \le 1$ we have
\begin{align*}
\E[(L^{(n)}(h)-L(h))^2] \le & 3\E[(L^{(n)}(h)-(m/2)L^{(n)}_{1/m}(h))^2] + 3(m/2)^{2}\E[(L^{(n)}_{1/m}(h)-L_{1/m}(h))^2] \\ &+ 3\E[(L(h)-(m/2)L_{1/m}(h))^2].    
\end{align*}
Let $c_m= \limsup_{n\to\infty} \E[(L^{(n)}(h)-(m/2)L^{(n)}_{1/m}(h))^2]$.
By Lemma \ref{lem:L-coupl-conv}, $\E[(L^{(n)}_{1/m}(h)-L_{1/m}(h))^2] \to 0$ as $n\to\infty$, since almost surely $L^{(n)}_{1/m}(h), L_{1/m}(h)\le 1$.
Thus we have
\[
\limsup_{n\to\infty}\E[(L^{(n)}(h)-L(h))^2] \le 3c_m + 3\E[(L(h)-(m/2)L_{1/m}(h))^2].
\]
Now we send $m\to\infty$ on the right-hand side. We have $\lim_{m\to\infty}c_m = 0$ by Proposition \ref{prop:Ln-conti}.
We also have that  $\lim_{m\to\infty}\E[(L(h)-(m/2)L_{1/m}(h))^2] = 0$, since $(L(h)-(m/2)L_{1/m}(h))^2 \to 0$ almost surely by Proposition \ref{prop:om-limit}, and these random variables are uniformly integrable in $m$, by Lemma \ref{lem:geo-ocp-tail}.
Thus we conclude that $\lim_{n\to\infty}\E[(L^{(n)}(h)-L(h))^2]=0$.
This implies that $L^{(n)}\to L$ in $[0,1]$, in finite-dimensional distributions.

We next upgrade the result to weak convergence in the topology of uniform convergence in $[0,1]$, using that $L^{(n)}$ and $L$ are non-decreasing.

For any $f:[0,1]\to\R$ and $k\in\N$, we denote by $M_k f$  the function where $(M_kf)(i/k)=f(i/k)$ for any $i\in \llbracket 0, k\rrbracket$, and linear between these points.
Then $M_kL^{(n)}$ converges to  $M_k L$ as $n\to\infty$ weakly in the topology of uniform convergence on compact sets, by the convergence of finite-dimensional distributions.
For any $k\in\N$, we let $D_k = \max_{i\in \llbracket 1,k\rrbracket} L(i/k)-L((i-1)/k)$, and $D_{n,k} = \max_{i\in \llbracket 1,k\rrbracket} L^{(n)}(i/k)-L^{(n)}((i-1)/k)$.
Then for any $h\in[0,1]$ we have $|M_kL(h)-L(h)|\le D_k$ and $|M_kL^{(n)}(h)-L^{(n)}(h)|\le D_{n,k}$, since $L^{(n)}$ and $L$ are non-decreasing.
Also by the convergence of finite-dimensional distributions, we have $D_{n,k}\to D_k$ in distribution as $n\to\infty$, for any $k\in \N$; and by Proposition \ref{prop:om-limit} we have $D_k\to 0$ in distribution as $k\to\infty$.
Thus the conclusion follows. 
\end{proof}

The remainder of this paper is devoted to proving Proposition \ref{prop:Ln-conti}.

\section{Continuity of occupation measure} \label{ssec:conti-ocp-mea}
In this section, we prove Proposition \ref{prop:Ln-conti}.
Without loss of generality, we assume $h=1$.

As already indicated, our general strategy is to write $L^{(n)}_w(1)$ as an average of local times of $\Gamma_\boo$, on different lines parallel to $\{(i,i):i \in \Z\}$; then we bound the differences of these local times (on different lines), in terms of a second moment bound.
More precisely, for each $n\in \N$ and $m\in\Z$, we define
\[
Z^{(n)}[m]= |\{0\le i \le n: (i+\lceil m/2 \rceil, i-\lfloor m/2 \rfloor) \in \Gamma_\boo\}|.
\]
Then we have
\begin{equation}  \label{eq:LwZz}
L^{(n)}(1)=2^{2/3}n^{-2/3}Z^{(n)}[0].
\end{equation}
For the function $L^{(n)}_w$, we have
\begin{equation}  \label{eq:LwZn}
2nL^{(n)}_w(1)=\sum_{|m|\le 2^{5/3}n^{2/3}w} Z^{(n)}[m] - \sum_{|2k+1| \le 2^{5/3}n^{2/3}w} \don[(n+k+1, n-k)\in \Gamma_\boo].
\end{equation}
The second term occurs due to a rounding issue from \eqref{eq:defn-Ln} and is at most $1$ and hence negligible.
We will prove the following result.
\begin{prop} \label{prop:bd-m-n}
We have $\lim_{w\to 0}\limsup_{n\to\infty}\max_{|m|\le wn^{2/3}}n^{-2/3}\E[(Z^{(n)}[m]-Z^{(n)}[0])^2] = 0$.
\end{prop}
As indicated in Section \ref{sec:iop}, we expect $\E[(Z^{(n)}[m]-Z^{(n)}[0])^2]$ to be of order $w^4n^{2/3}$ when $|m|$ is approximately $wn^{2/3}$.
However, Proposition \ref{prop:bd-m-n} will already suffice for our application, i.e., proving Proposition \ref{prop:Ln-conti}.
\begin{proof}[Proof of Proposition \ref{prop:Ln-conti}]
By \eqref{eq:LwZz} and \eqref{eq:LwZn}, we have
\[\left|2nL^{(n)}_w(1) - \sum_{|m|\le 2^{5/3}n^{2/3}w} Z^{(n)}[m]\right| \le 1.\]
Thus we have
\begin{align*}
\left|(2w)^{-1}L^{(n)}_w(1)-L^{(n)}(1)\right| \le & 
 \left|(4nw)^{-1}\sum_{|m|\le 2^{5/3}n^{2/3}w} Z^{(n)}[m] - 2^{2/3}n^{-1/3}Z^{(n)}[0]\right|  + (4nw)^{-1}\\
\le & 
\left|\left( (4nw)^{-1} - \frac{2^{2/3}n^{-1/3}}{2\lfloor 2^{5/3}n^{2/3}w\rfloor +1}\right)\sum_{|m|\le 2^{5/3}n^{2/3}w} Z^{(n)}[m]\right| \\
&+
\left|\sum_{|m|\le 2^{5/3}n^{2/3}w} \frac{2^{2/3}n^{-1/3}}{2\lfloor 2^{5/3}n^{2/3}w\rfloor +1}(Z^{(n)}[m]-Z^{(n)}[0])\right| +(4nw)^{-1}.
\end{align*}
For the term in the second line, since $\sum_{m\in\Z} Z^{(n)}[m]\le 2n+2$, we have
\[
\left|\left( (4nw)^{-1} - \frac{2^{2/3}n^{-1/3}}{2\lfloor 2^{5/3}n^{2/3}w\rfloor +1}\right)\sum_{|m|\le 2^{5/3}n^{2/3}w} Z^{(n)}[m]\right| \le n^{-2/3}w^{-2}.
\]
For the term in the final line, by Cauchy-Schwarz inequality, we have
\begin{align*}
&\E\left[ \left(
\sum_{|m|\le 2^{5/3}n^{2/3}w} \frac{2^{2/3}n^{-1/3}}{2\lfloor 2^{5/3}n^{2/3}w\rfloor +1}(Z^{(n)}[m]-Z^{(n)}[0]) 
\right)^2 \right]\\ \le &
\sum_{|m|\le 2^{5/3}n^{2/3}w} \frac{2^{4/3}n^{-2/3}}{2\lfloor 2^{5/3}n^{2/3}w\rfloor +1}\E\left[(Z^{(n)}[m]-Z^{(n)}[0])^2\right],     
\end{align*}
which decays to $0$ as $n\to \infty$ then $w\to 0$, by Proposition \ref{prop:bd-m-n}.
Thus the conclusion follows by combining the above two estimates.
\end{proof}

It remains to prove Proposition \ref{prop:bd-m-n}.
The idea is to decompose $\E[(Z^{(n)}[m]-Z^{(n)}[0])^2]$ into sums of two-point functions of $\Gamma_\boo$. The next step then involves establishing a certain smoothness property of the two-point function as the points move about in the spatial direction. 
It will suffice to simply consider the case when one point remains the same, and another point moves by $m$ in the spatial direction.
In this direction, the following is the key estimate.
Recall that $d(x,y)=x+y$ and $ad(x,y)=x-y$ for any $(x,y)\in \Z^2$.
\begin{prop} \label{prop:occup-diff}
For any $0<\theta, w <1$, take any $p_1, p_2, p_3 \in \Z^2$ such that 
\begin{enumerate}
    \item $\theta n\le d(p_1) < d(p_2)\le d(p_3) \le 2n+1$,
    \item $d(p_2)-d(p_1)>\theta n$, $d(p_3)-d(p_2)\le 1$,
    \item $|ad(p_1)|, |ad(p_2)|, |ad(p_3)| \le wn^{2/3}$.
\end{enumerate}
Then, when $n$ is large enough (depending on $w, \theta$) we have $|\PP[p_1, p_2 \in \Gamma_\boo] - \PP[p_1, p_3 \in \Gamma_\boo]| < c_{w,\theta}n^{-4/3}$, where $c_{w,\theta}$ depends only on $w, \theta$, and $\lim_{w\to 0}c_{w,\theta} = 0$ for any $\theta$.
\end{prop}
The heuristic behind the upper bound $c_{w,\theta}n^{-4/3}$ is as follows. 
Since the spatial fluctuation of $\Gamma_\boo$ is in the order of $n^{2/3}$, we expect that each of $\PP[p_1 \in \Gamma_\boo]$, $\PP[p_2 \in \Gamma_\boo]$, and $\PP[p_3 \in \Gamma_\boo]$  is in the order of $n^{-2/3}$.
Further, the separation lower bound on $d(p_2)-d(p_1)>\theta n$ implies a degree of independence leading to both $\PP[p_1, p_2 \in \Gamma_\boo]$ and $\PP[p_1, p_2 \in \Gamma_\boo]$ being in the order of $n^{-4/3}$. The fact that the distance between $p_2$ and $p_3$ is at most $2wn^{2/3}$ produces the factor of $c_{w,\theta}$.

Postponing the proof of Proposition \ref{prop:occup-diff} till later, momentarily we prove Proposition \ref{prop:bd-m-n} assuming it.

\begin{figure}
     \centering
     \begin{subfigure}[t]{0.30\textwidth}
         \centering
\begin{tikzpicture}
[line cap=round,line join=round,>=triangle 45,x=0.4cm,y=0.4cm]
\clip(-1,-1) rectangle (11,11);

\draw [line width=0.05pt] [blue] (0,0) -- (1.6,0.4) -- (1.7,2.3) -- (3,3) -- (4.9,3.1) -- (5.1,4.9) -- (6.8,5.2) -- (7.5,6.5) -- (7.9,8.1) -- (9.2,8.8) -- (9.3,10.7) -- (11.2,10.8);

\begin{tiny}
\draw (0,0) node[anchor=south east]{$\boo$};
\draw (3,3) node[anchor=south east]{$(i,i)$};
\draw (7,7) node[anchor=south east]{$(j,j)$};
\draw (3.5,2.5) node[anchor=north west]{$(i+\lceil m/2\rceil,i-\lfloor m/2\rfloor)$};
\draw (7.5,6.5) node[anchor=north]{$(j+\lceil m/2\rceil,j-\lfloor m/2\rfloor)$};
\draw [blue] (9.3,10.7) node[anchor=east]{$\Gamma_\boo$};
\end{tiny}

\draw [ultra thin] (-10,-10) -- (12,12);
\draw [ultra thin] (-10,-11) -- (11,10);

\draw [fill=uuuuuu] (0,0) circle (1.5pt);

\draw [fill=red, color=red] (3,3) circle (1.5pt);
\draw [fill=red, color=red] (3.5,2.5) circle (1.5pt);
\draw [fill=red, color=red] (7,7) circle (1.5pt);
\draw [fill=red, color=red] (7.5,6.5) circle (1.5pt);

\end{tikzpicture}
     \end{subfigure}
     \hfill
     \begin{subfigure}[t]{0.30\textwidth}
         \centering
\begin{tikzpicture}
[line cap=round,line join=round,>=triangle 45,x=0.4cm,y=0.4cm]
\clip(-1,-1) rectangle (11,11);

\draw [line width=0.05pt] [blue] (0,0) -- (1.6,0.4) -- (1.7,2.3) -- (3,3) -- (4.5,3.5) -- (5.1,4.9) -- (6.8,5.2) -- (7.5,6.5) -- (7.9,8.1) -- (9.2,8.8) -- (9.3,10.7) -- (11.2,10.8);

\begin{tiny}
\draw (0,0) node[anchor=south east]{$\boo$};
\draw (3,3) node[anchor=south east]{$(i,i)$};
\draw (4,4) node[anchor=south east]{$(j,j)$};
\draw (3.5,2.5) node[anchor=north west]{$(i+\lceil m/2\rceil,i-\lfloor m/2\rfloor)$};
\draw (4,3.5) node[anchor=north west]{$(j+\lceil m/2\rceil,j-\lfloor m/2\rfloor)$};
\draw [blue] (9.3,10.7) node[anchor=east]{$\Gamma_\boo$};
\end{tiny}

\draw [ultra thin] (-10,-10) -- (12,12);
\draw [ultra thin] (-10,-11) -- (11,10);

\draw [fill=uuuuuu] (0,0) circle (1.5pt);

\draw [fill=red, color=red] (3,3) circle (1.5pt);
\draw [fill=red, color=red] (3.5,2.5) circle (1.5pt);
\draw [fill=red, color=red] (4,4) circle (1.5pt);
\draw [fill=red, color=red] (4.5,3.5) circle (1.5pt);

\end{tikzpicture}
     \end{subfigure}
     \hfill
     \begin{subfigure}[t]{0.30\textwidth}
         \centering
\begin{tikzpicture}
[line cap=round,line join=round,>=triangle 45,x=0.4cm,y=0.4cm]
\clip(-1,-1) rectangle (11,11);

\draw [line width=0.05pt] [blue] (0,0) -- (1.5,0.5) -- (1.7,2.3) -- (3,3) -- (4.9,3.1) -- (5.1,4.9) -- (6.8,5.2) -- (7.5,6.5) -- (7.9,8.1) -- (9.2,8.8) -- (9.3,10.7) -- (11.2,10.8);

\begin{tiny}
\draw (0,0) node[anchor=south east]{$\boo$};
\draw (1,1) node[anchor=south east]{$(i,i)$};
\draw (7,7) node[anchor=south east]{$(j,j)$};
\draw (1.5,0.5) node[anchor=north west]{$(i+\lceil m/2\rceil,i-\lfloor m/2\rfloor)$};
\draw (7.5,6.5) node[anchor=north]{$(j+\lceil m/2\rceil,j-\lfloor m/2\rfloor)$};
\draw [blue] (9.3,10.7) node[anchor=east]{$\Gamma_\boo$};
\end{tiny}

\draw [ultra thin] (-10,-10) -- (12,12);
\draw [ultra thin] (-10,-11) -- (11,10);

\draw [fill=uuuuuu] (0,0) circle (1.5pt);

\draw [fill=red, color=red] (1,1) circle (1.5pt);
\draw [fill=red, color=red] (1.5,0.5) circle (1.5pt);
\draw [fill=red, color=red] (7,7) circle (1.5pt);
\draw [fill=red, color=red] (7.5,6.5) circle (1.5pt);

\end{tikzpicture}
     \end{subfigure}
        \caption{An illustration of the proof of Proposition \ref{prop:bd-m-n}: we decompose the second moment of $Z^{(n)}[m]-Z^{(n)}[0]$ as the sum of two-point probabilities of $\Gamma_\boo$. To apply Proposition \ref{prop:occup-diff}, we further divide the sum over the two points into three parts, which are respectively illustrated in this figure.}
        \label{fig:3e}
\end{figure}
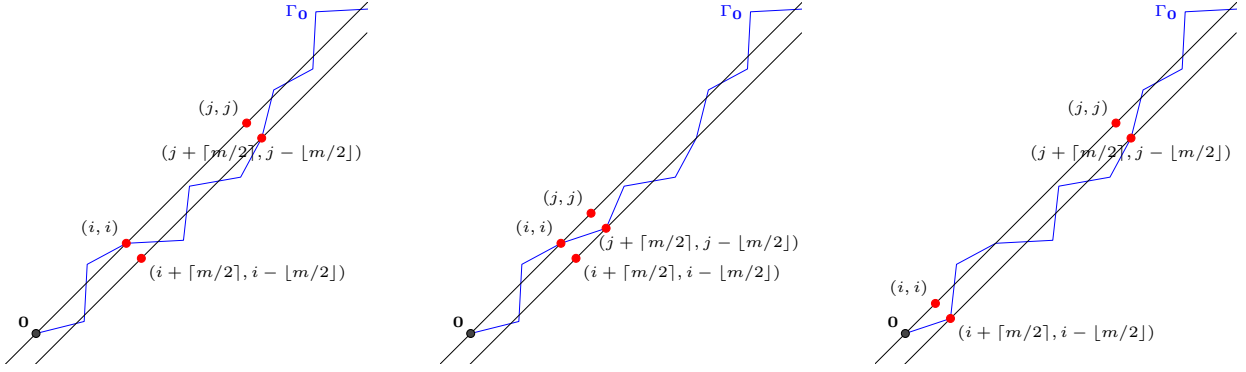

\begin{proof}[Proof of Proposition \ref{prop:bd-m-n}]In this proof we use $C,c>0$ to denote large and small constants.

We can write
\begin{align*}
\E[(Z^{(n)}[0]-Z^{(n)}[m])^2] =& \E[(Z^{(n)}[0])^2] - 2\E[Z^{(n)}[0]Z^{(n)}[m]] + \E[(Z^{(n)}[m])^2]
\\
=& \sum_{i,j=0}^n 
\Big(\PP[(i,i), (j,j) \in \Gamma_\boo] -
2\PP[(i,i), (j+\lceil m/2 \rceil, j-\lfloor m/2 \rfloor)\in \Gamma_\boo] \\ &+
\PP[(i+\lceil m/2 \rceil, i-\lfloor m/2 \rfloor), (j+\lceil m/2 \rceil, j-\lfloor m/2 \rfloor)\in \Gamma_\boo]\Big).
\end{align*}

For the sum on $i, j$ from $0$ to $n$, Proposition \ref{prop:occup-diff} only applies to the case where both $i, j$ are far from $0$, and $|i-j|$ is not small.
There are two other parts in the sum: where $i$ and $j$ are close to each other, and where $i$ or $j$ is near $0$.
For these two cases, we just bound the sum of each probability (instead of the differences), using Lemma \ref{lem:disc-max} and translation invariant arguments.

More precisely, we consider the following three parts (see Figure \ref{fig:3e}):
\begin{enumerate}
    \item $\theta n/2\le i,j\le n, |i-j|>\theta n/2$,
    \item $0\le i,j\le n, |i-j|\le \theta n/2$,
    \item $0\le i \le \theta n /2, 0\le j\le n$ or $0\le i \le n, 0\le j\le \theta n /2$.
\end{enumerate}
For Part 1, we use Proposition \ref{prop:occup-diff}. 
Indeed, take any $0<\theta, w<1$.
Since $|m|\le wn^{2/3}$, we have
\begin{align*}
&\sum_{\theta n/2\le i,j\le n, |i-j|>\theta n/2}\Big(\PP[(i,i), (j,j) \in \Gamma_\boo] -
2\PP[(i,i), (j+\lceil m/2 \rceil, j-\lfloor m/2 \rfloor)\in \Gamma_\boo] \\ &+
\PP[(i+\lceil m/2 \rceil, i-\lfloor m/2 \rfloor), (j+\lceil m/2 \rceil, j-\lfloor m/2 \rfloor)\in \Gamma_\boo]\Big)
\\
< & 2c_{w,\theta} n^{2/3},
\end{align*}
since by Proposition \ref{prop:occup-diff}, each summand is bounded by $2c_{w,\theta}n^{-4/3}$.

We next bound the sums for $i, j$ in Part 2 and Part 3.
We write the sums as the number of points in the intersection of $\Gamma_\boo$ with some short segments of length $\theta n$, and use the bound of Lemma \ref{lem:disc-max}.
For Part 3, the only such short segment we would consider is the one near the origin. For Part 2, we would consider short segments far from the origin as well, and we would need to also use the fact that such segments would be disjoint from $\Gamma_\boo$ with high probability.
Eventually, we will bound the sums over $i,j$ in Part 2 and Part 3 by
$C\log(\theta^{-1})^4\theta^{1/3} n^{2/3}$ and $C\theta^{1/3} n^{2/3}$, respectively. 
Combining this with the above bound for Part 1, we conclude that
\[
n^{-2/3}\E[(Z^{(n)}[m]-Z^{(n)}[0])^2] \le 2c_{w,\theta}  + C\log(\theta^{-1})^4\theta^{1/3}.
\]
By sending $w\to 0$, and then noting that $\theta>0$ is arbitrary, the conclusion follows.

We now finish the bounds for Part 2 and Part 3.
For this we let $w<\theta^{2/3}$. 
For each $k\in \llbracket 1, \lceil 2/\theta \rceil \rrbracket$, and $m\in\Z$, $|m|\le wn^{2/3}$, we define
\[
Z^{(n,k)}[m]= |\{(k-1)\theta n/2 \le i \le (k+1)\theta n/2 : (i+\lceil m/2 \rceil, i-\lfloor m/2 \rfloor) \in \Gamma_\boo\}|.
\]
\noindent\textbf{Part 2:}
We have
\begin{equation} \label{eq:i-j-close}
\begin{split}
&\sum_{0\le i,j\le n, |i-j|\le \theta n/2}\Big(\PP[(i,i), (j,j) \in \Gamma_\boo] \\ &+
\PP[(i+\lceil m/2 \rceil, i-\lfloor m/2 \rfloor), (j+\lceil m/2 \rceil, j-\lfloor m/2 \rfloor)\in \Gamma_\boo]\Big)
\\
\le &\sum_{k=1}^{\lceil 2/\theta \rceil} \E[(Z^{(n,k)}[0])^2] + \E[(Z^{(n,k)}[m])^2].
\end{split}
\end{equation}
We next bound each $\E[(Z^{(n,k)}[m])^2]$.
The general idea is to use Lemma \ref{lem:disc-max}; but to get the desired bound, we also need to bound $\PP[Z^{(n,k)}[m]> 0]$, since $Z^{(n,k)}[m]=0$ with high probability, when $k$ is large.
In other words, we need to bound the probability that $\Gamma_\boo$ intersects a line segment, which is in the $(1,1)$-direction and has length $\theta n$.
Our strategy is to reduce this to the probability that $\Gamma_\boo$ intersects a line segment in the $(1,-1)$-direction with length of order $(\theta n)^{2/3}$, using transversal fluctuation estimates of geodesics. 

Take $j_* \in \Z$ such that $(\lceil (k-1)\theta n/2 \rceil + j_*, \lceil (k-1)\theta n/2\rceil -j_*) \in \Gamma_\boo$. 
Denote \[p_{*,1} = (\lceil (k-1)\theta n/2 \rceil + \lfloor m/2 + \log(k)^2(\theta n)^{2/3} \rfloor, \lceil (k-1)\theta n/2\rceil -\lfloor m/2 + \log(k)^2(\theta n)^{2/3} \rfloor),\]\[p_{*,2} = (\lceil (k-1)\theta n/2 \rceil + \lfloor m/2 - \log(k)^2(\theta n)^{2/3} \rfloor, \lceil (k-1)\theta n/2\rceil -\lfloor m/2 - \log(k)^2(\theta n)^{2/3} \rfloor).\]
The line segment connecting $p_{*,1}$ and $p_{*,2}$ is the $(1,-1)$-direction line segment that we will use.
If $Z^{(n,k)}[m]> 0$, and $(\lceil (k-1)\theta n/2 \rceil + j_*, \lceil (k-1)\theta n/2\rceil -j_*)$ is not contained in the line segment connecting $p_{*,1}$ and $p_{*,2}$, then (by the tree structure of semi-infinite geodesics) one of $\Gamma_{p_{*,1}}$ and $\Gamma_{p_{*,2}}$ must intersect the set $\{(i+\lceil m/2 \rceil, i-\lfloor m/2 \rfloor) : (k-1)\theta n/2 \le i \le (k+1)\theta n/2\}$.
We thus have 
\begin{multline*}
\PP[Z^{(n,k)}[m]> 0] \le 
\PP[ |j_* - m/2| < 2\log(k)^2(\theta n)^{2/3} ] \\ + \PP[(\Gamma_{p_{*,1}} \cup\Gamma_{p_{*,2}})\cap \{(i+\lceil m/2 \rceil, i-\lfloor m/2 \rfloor) : (k-1)\theta n/2 \le i \le (k+1)\theta n/2\} \neq \emptyset].    
\end{multline*}
By Lemma \ref{lem:trans-dis-inf}, the last line is bounded by $Ce^{-c\log(k)^6}$.
By translation invariance we have
\[
\PP[ |j_* - m/2| < 2\log(k)^2(\theta n)^{2/3} ] < \frac{2\log(k)^2(\theta n)^{2/3}}{(k\theta n)^{2/3}} \E[X],
\]
where 
\[X=|\{p\in \LL_{2\lceil (k-1)\theta n/2 \rceil}: |ad(p)|\le 2(k\theta n)^{2/3}+|m|+4\log(k)^2(\theta n)^{2/3}, p \in \cup_{|i| < (k\theta n)^{2/3}} \Gamma_{(i,-i)} \}|.
\]
Note that $|m|\le wn^{2/3} < (\theta n)^{2/3}$.
By Lemma \ref{lem:num-int-dis-inf} we have $\E[X]< C$. Thus we conclude that 
\[\PP[Z^{(n,k)}[m]>0]< C\log(k)^2k^{-2/3}.\]
On the other hand, by Lemma \ref{lem:disc-max} we also have that $\PP[Z^{(n,k)}[m]>M(\theta n)^{1/3}]<Ce^{-cM}$, for any $M>0$.
Therefore we have
\[
\begin{split}
\E[(Z^{(n,k)}[m])^2] &= (\theta n)^{2/3} \int_0^\infty 2M\PP[Z^{(n,k)}[m]>Mk^{-1/3}(\theta n)^{1/3}] dM 
\\
&\le C(\theta n)^{2/3} \int_0^\infty M(\log(k)^2k^{-2/3} \wedge e^{-cM}) dM
\\
& \le C\log(k)^4 k^{-2/3}(\theta n)^{2/3}.
\end{split}
\]
For $k=1$ we have that $\E[(Z^{(n,1)}[m])^2]\le C(\theta n)^{2/3}$ by Lemma \ref{lem:disc-max}.
Thus we can bound the last line of \eqref{eq:i-j-close} by $C\log(\theta^{-1})^4\theta^{1/3} n^{2/3}$, as desired.

\noindent\textbf{Part 3:}
We have that
\begin{equation} \label{eq:i-small}
\begin{split}
&\sum_{0\le i \le \theta n /2, 0\le j\le n}\Big(\PP[(i,i), (j,j) \in \Gamma_\boo] \\ &+
\PP[(i+\lceil m/2 \rceil, i-\lfloor m/2 \rfloor), (j+\lceil m/2 \rceil, j-\lfloor m/2 \rfloor)\in \Gamma_\boo]\Big)
\\
\le & \E[Z^{(n,1)}[0]Z^{(n)}[0]] + \E[Z^{(n,1)}[m]Z^{(n)}[m]]
\\
\le
& \sqrt{\E[(Z^{(n,1)}[0])^2]\E[(Z^{(n)}[0])^2]} + \sqrt{\E[(Z^{(n,1)}[m])^2]\E[(Z^{(n)}[m])^2]}.
\end{split}
\end{equation}
By Lemma \ref{lem:disc-max}, $\E[(Z^{(n)}[m])^2] \le C n^{2/3}$ and $\E[(Z^{(n,1)}[m])^2]\le C(\theta n)^{2/3}$.
Thus we can bound the last line of \eqref{eq:i-small} by $C\theta^{1/3} n^{2/3}$, as desired. 
\end{proof}

\subsection{Smoothness of the two-point distribution}  \label{ssec:diff-two-point-prob}
The rest of this section is devoted to proving Proposition \ref{prop:occup-diff}, 
the key estimate on the two-point distribution of the geodesic $\Gamma_\boo$.

We now explain the general strategy.
We cut $\Z^2$ into two parts by a line $\LL_m$, which separates $p_1$ and $p_2, p_3$.
For the random field $\{\omega_p:p\in\Z^2\}$, we translate the part above $\LL_m$, sending $p_2$ to $p_3$.
It then remains to show that, given $p_1, p_2 \in \Gamma_\boo$, with a high probability this translation would not affect the local structures of $\Gamma_\boo$, and hence after the translation, we have $p_1, p_3 \in \Gamma_\boo$ (see Figure \ref{fig:42}, left panel).

\begin{figure}[t]
         \centering
     \begin{subfigure}[t]{0.7\textwidth}
         \centering
\begin{tikzpicture}
[line cap=round,line join=round,>=triangle 45,x=0.7cm,y=0.7cm]
\clip(-3,-3) rectangle (11,11);

\fill[line width=0.pt,color=green,fill=green,fill opacity=0.15]
(2,12) -- (12,2) -- (12,12) -- cycle;

\draw [line width=1.2pt] [blue]  (6.7,7.3) -- (7.7,8.3) -- (9.2,8.8) -- (9.3,10.7) -- (11.2,10.8);
\draw [line width=0.7pt] [brown] (7,7) -- (8.2,7.8) -- (9.7,8.3) -- (9.8,10.2) -- (11.7,10.3);

\draw [line width=1.2pt] [blue] (-2,-2) -- (-1.4,-0.6) -- (0.2,-0.2) -- (1.5,0.5) -- (1.7,2.3) -- (3,3) -- (3.6,3.4);
\draw [line width=0.7pt] [brown] (-2,-2) -- (-1.4,-0.6) -- (0.2,-0.2) -- (1.5,0.5) -- (1.7,2.3) -- (3,3) -- (3.6,3.4);

\draw [line width=1.2pt] [blue] (3.6,3.4) -- (4.3,3.7) -- (4.4,4.6) -- (5.1,4.9) -- (5.4,5.6) -- (6,6) -- (6.2,6.8) -- (6.7,7.3);
\draw [line width=0.7pt] [brown] (3.6,3.4) -- (4.3,3.7) -- (4.4,4.6) -- (5.1,4.9) -- (5.4,5.6) -- (6.2,5.8) -- (6.4,6.6) -- (7,7);

\begin{scriptsize}
\draw (1.7,2.3) node[anchor=south east]{$p_1$};
\draw (9.2,8.8) node[anchor=south east]{$p_2$};
\draw (9.7,8.3) node[anchor=north west]{$p_3$};

\draw [blue] (9.3,10.7) node[anchor=east]{$\Gamma_\boo$};
\draw [brown] (9.8,10.2) node[anchor=north west]{$\Gamma_\boo'$};

\draw (-2,-2) node[anchor=north east]{$\boo$};
\draw (11,3) node[anchor=north east]{$\LL_m$};
\draw (9,-2) node[anchor=north east]{$\LL_{m_-}$};
\end{scriptsize}

\draw [fill=uuuuuu] (-2,-2) circle (1.5pt);

\draw [fill=red, color=red] (1.7,2.3) circle (1.5pt);
\draw [fill=red, color=red] (9.2,8.8) circle (1.5pt);
\draw [fill=red, color=red] (9.7,8.3) circle (1.5pt);

\draw [line width=0.4pt] [blue]  plot [smooth] coordinates {(5.5,11.5) (5,10.3) (6.3,9.8) (6.3,8.2) (6.7,7.3) (8,7.7) (8.5,6) (10,5.7)};
\draw [line width=0.4pt] [brown]  plot [smooth] coordinates {(6,11) (5.5,9.8) (6.8,9.3) (6.8,7.7) (7.2,6.8) (8.5,7.2) (9,5.5) (10.5,5.2)};

\draw [line width=0.4pt] [blue]  plot [smooth] coordinates {(-0.3,4.6) (1,5.1) (1.2,3.6) (2.8,3.7) (3.6,3.4) (3.4,2) (5,1.4) (5.3,-0.1)};
\draw [line width=0.2pt] [brown]  plot [smooth] coordinates {(-0.3,4.6) (1,5.1) (1.2,3.6) (2.8,3.7) (3.6,3.4) (3.4,2) (5,1.4) (5.3,-0.1)};

\draw [line width=1pt] (2,5) -- (5,2);
\draw [line width=1pt] (5.5,8.5) -- (8.5,5.5);
\draw [dotted] [thick] (0,14) -- (14,0);
\draw [dotted] [thick] (-5,12) -- (12,-5);

\end{tikzpicture}
\end{subfigure}
\hfill
     \begin{subfigure}[t]{0.28\textwidth}
         \centering
         \begin{tikzpicture}
[line cap=round,line join=round,>=triangle 45,x=1.5cm,y=1.5cm]
\clip(0,-0.2) rectangle (4,7.2);

\draw [line width=0.3pt]  (3.1,0.3) -- (0.3,3.1);
\draw [line width=0.3pt]  (3.1,4.3) -- (0.3,7.1);

\foreach \i in {1,...,15}
{
\foreach \j in {1,...,15}
\draw [fill=uuuuuu, color=uuuuuu] (\i/5,\j/5) circle (0.8pt);
}

\foreach \i in {2,...,15}
{
\foreach \j in {2,...,\i}
\draw [fill=blue, color=blue] (\i/5,3.4-\j/5) circle (1.2pt);
}

\foreach \i in {3,...,15}
{
\foreach \j in {3,...,\i}
\draw [fill=brown, color=brown] (\i/5,3.6-\j/5) circle (0.6pt);
}

\foreach \i in {1,...,15}
{
\foreach \j in {1,...,15}
\draw [fill=uuuuuu, color=uuuuuu] (\i/5,\j/5+4) circle (0.8pt);
}

\foreach \i in {2,...,15}
{
\foreach \j in {2,...,\i}
\draw [fill=blue, color=blue] (\i/5,3.4-\j/5+4) circle (1.2pt);
}

\foreach \i in {2,...,15}
{
\foreach \j in {2,...,\i}
\draw [fill=brown, color=brown] (\i/5,3.4-\j/5+4) circle (0.6pt);
}

\draw [thick] [red] [-latex] (0.6,6.8) to [out=-30,in=120] (1.8,5.6);
\draw [thick] [red] [-latex] (0.8,6.8) to [out=-30,in=120] (2.0,5.6);
\draw [thick] [red] [-latex] (1.0,6.8) to [out=-30,in=120] (2.2,5.6);

\draw [thick] [red] [-latex] (0.6,2.8) to (1.8,1.8);
\draw [thick] [red] [-latex] (0.8,2.8) to (2.0,1.8);
\draw [thick] [red] [-latex] (1.0,2.8) to (2.2,1.8);

\begin{scriptsize}
\draw (3.1,0.3) node[anchor=east]{$\LL_m$};
\draw (3.1,4.3) node[anchor=east]{$\LL_m$};

\draw (1.5,0.1) node[anchor=north]{$d(p_3)=d(p_2)+1$};
\draw (1.5,4.1) node[anchor=north]{$d(p_3)=d(p_2)$};
\end{scriptsize}

\end{tikzpicture}
    \end{subfigure}
        \caption{An illustration of the proof of Proposition \ref{prop:occup-diff}: by shifting the random field above $\LL_m$, it is likely that the geodesics $\Gamma_\boo$ and $\Gamma_\boo'$ are the same below $\LL_{m_-}$ (left panel).
        The right panel is on the details of shifting the random field in two cases. The random field at the blue vertices is shifted to the random field at the brown vertices. In the case of $d(p_3)=d(p_2)+1$, the random variables $\{\omega_p': d(p)=2m\}$ are independent of $\{\omega_p: p\in\Z^2\}$. }
        \label{fig:42}
\end{figure}
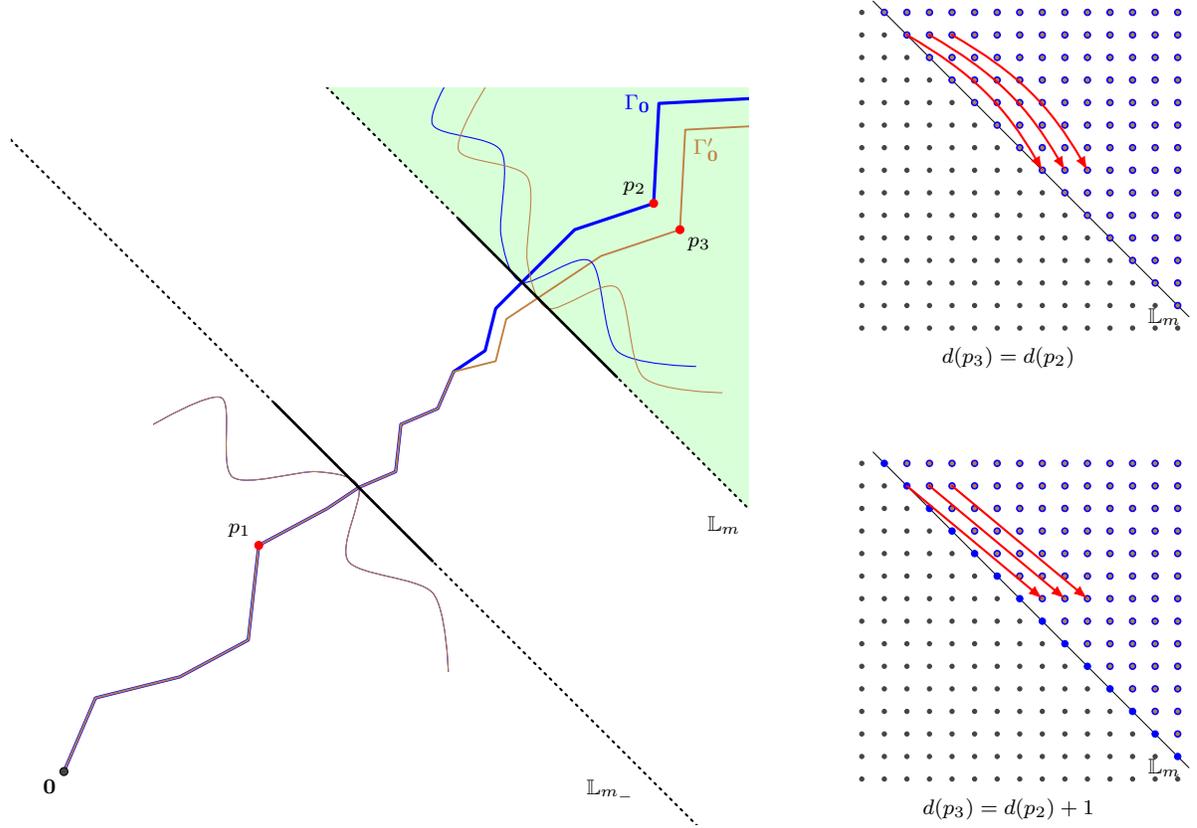

To get the desired stability of the local structures, we will use properties of the directed landscape, via the limiting transition of Theorem \ref{thm:exp-to-dl}.
This way we would obtain results for $m/n$ converging to a constant.
However, our ultimate goal is to get estimates that are uniform over $p_1, p_2, p_3$ satisfying the three conditions in the statement of this proposition.
To address this issue, we will take $m$ among a mesh of numbers, whose cardinality does not grow as $n\to\infty$, while covering all possible choices of $p_1, p_2, p_3$.
\begin{proof}[Proof of Proposition \ref{prop:occup-diff}]
Take $m=\lfloor \kappa\theta n/100 \rfloor$ for some $\kappa \in \llbracket 1, 100\theta^{-1}\rrbracket$, such that
$d(p_1)+\theta n /3 < 2m < d(p_2)-\theta n/3$.
Note that the number of possible values of $m$ is $\lfloor 100\theta^{-1} \rfloor$. 
Such $m$ always exists due to the first two conditions on $p_1, p_2, p_3$ in the statement of the proposition.

We consider another random field $\{\omega_p': p \in \Z^2\}$ of i.i.d.\ $\Exp(1)$ random variables, coupled with the original field $\{\omega_p: p\in\Z^2\}$, by letting $\omega_p = \omega_p'$ for each $p\in\Z^2, d(p)< 2m$, and $\omega_p = \omega_{p+p_3-p_2}'$ for each $p\in\Z^2, d(p)\ge 2m$.
In the case where $d(p_3)=d(p_2)+1$, the weights $\{\omega_p': d(p)=2m\}$ remain undetermined, and we let them be i.i.d.\ $\Exp(1)$ and independent of $\{\omega_p: p\in\Z^2\}$ (see Figure \ref{fig:42}, right panel).
For each $q\in\Z^2$ we let $\Gamma_q'$ be the $(1,1)$-direction semi-infinite geodesic in the field $\{\omega_p': p\in\Z^2\}$.

We then have that
\begin{align*}
\PP[p_1, p_2 \in \Gamma_\boo] - \PP[p_1, p_3 \in \Gamma_\boo]&=\PP[p_1, p_2 \in \Gamma_\boo] - \PP[p_1, p_3 \in \Gamma'_\boo]\\
&\le \PP[p_1, p_2 \in \Gamma_\boo, \{p_1, p_3\} \not\subset \Gamma_\boo']
\\
&\le
\PP[p_1, p_2 \in \Gamma_\boo, p_1 \not\in \Gamma_\boo']
+
\PP[p_1, p_2 \in \Gamma_\boo, p_3 \not\in \Gamma_\boo'].    
\end{align*}

We next bound $\PP[p_1, p_2 \in \Gamma_\boo, p_1 \not\in \Gamma_\boo']$. The bound for the other term follows the same arguments. 
The general idea is to decompose the event into independent events whose probabilities are easier to control.

Take $m_- = \lfloor \kappa_-\theta n/100 \rfloor$ for some $\kappa_- \in \llbracket 1, 100\theta^{-1}\rrbracket$, such that
$d(p_1)+\theta n /7 < 2m_- < 2m-\theta n/7$.
In words, $2m_-$ is order $\theta n$ separated from $d(p_1)$ and $2m$, as illustrated in Figure \ref{fig:42} (left panel).
Like $m$, this $m_-$ is also chosen among a finite mesh of numbers.

Note that on and below $\LL_{m_-+1}$, both $\Gamma_\boo$ and $\Gamma_\boo'$ are geodesics in the random field $\{\omega_p:p\in\Z^2\}$.
Thus $p_1 \in \Gamma_\boo, p_1 \not\in \Gamma_\boo'$ would imply that $\Gamma_\boo \cap \LL_{m_-+1} \neq \Gamma_\boo'\cap \LL_{m_-+1}$ (see Figure \ref{fig:42}, left panel).
It now suffices to bound 
\begin{equation}  \label{eq:p1p2e3}
\PP[p_1, p_2 \in \Gamma_\boo,\; \Gamma_\boo \cap \LL_{m_-+1} \neq \Gamma_\boo'\cap \LL_{m_-+1}].
\end{equation}
We would like to exploit the fact that owing to their separation, the three events $p_1 \in \Gamma_\boo$, $p_2 \in \Gamma_\boo$, and $\Gamma_\boo \cap \LL_{m_-+1} \neq \Gamma_\boo'\cap \LL_{m_-+1}$ mostly rely on some local randomness: below $\LL_{m_-}$, above $\LL_m$, or between these two lines which can be used to construct independent proxy events. We design these next.

Take $\beta>0$ to be a small constant, which will be used to quantify the regularities of (1) the passage times from $\boo$ to $\LL_{m_-}$, and (2) the Busemann function along $\LL_m$.
\begin{enumerate}
    \item Define $\ocE_1$ to be the event that $p_1 \in \cup_{|i|\le m_-} \Gamma_{\boo, (m_--i,m_-+i)}$. A reason behind this definition is that this event is measurable with respect to $\{\omega_p: d(p)\le 2m_-\}$, the random field on and below $\LL_{m_-}$. Let $\cE_1$ be the event that $\ocE_1$ holds along with the following regularity conditions:
  \begin{multline}  \label{eq:betareg}
       |T_{\boo, (m_--i,m_-+i)} - T_{\boo, (m_--j,m_-+j)}| < \beta^{-2}|i-j|^{1/2}(n^{2/3}/|i-j|)^{0.01},\\ \forall |i|, |j|<\beta^{-1}n^{2/3}, wn^{2/3}<|i-j|<n^{2/3}, 
    \end{multline}
    \begin{equation}  \label{eq:betarege}
    |T_{\boo, (m_--i,m_-+i)}-\E[T_{\boo, (m_--i,m_-+i)}]| < \beta^{-1}n^{1/3} + \beta i^2/n, \quad \forall i\in\Z.        
    \end{equation}
    Then $\cE_1$ is also measurable with respect to $\{\omega_p: d(p)\le 2m_-\}$.
    For the passage times from $\boo$ to $\LL_{m_-}$, the local fluctuation should actually exhibit random walk behavior (similar to how the Airy$_2$ process is locally Brownian like as stated in Section \ref{sssec:atlocbr}; this is also indicated by the random walk behavior of the Busemann function stated in Section \ref{ssec:busemann}), so the left-hand side of \eqref{eq:betareg} should be of order $|i-j|^{1/2}$. For the bound to hold uniformly in $i$ and $j$, an extra factor of $(n^{2/3}/|i-j|)^{0.01}$ is added to the right-hand side.
    For \eqref{eq:betarege}, the term $\beta i^2/n$ is added to make this bound hold for all $i\in\Z$. 
    As $\beta$ is small, \eqref{eq:betarege} ensures that $T_{\boo, (m_--i,m_-+i)}$ decays as $|i|$ increases, since the expectation $\E[T_{\boo, (m_--i,m_-+i)}]$ already has a parabolic decay in $i$ (see \eqref{e:mean}).
    \item Similarly, let $\ocE_2$ be the event where $p_2 \in \cup_{|i|\le m} \Gamma_{(m-i,m+i)}$. 
 Let $\cE_2$ be the event that $\ocE_2$ as well as the following regularity conditions hold:
    \begin{multline*}
        |G((m-i,m+i)) - G((m-j,m+j))| < \beta^{-2}|i-j|^{1/2}(n^{2/3}/|i-j|)^{0.01},\\ \forall |i|, |j|<\beta^{-1}n^{2/3}, wn^{2/3}<|i-j|<n^{2/3}, 
    \end{multline*}
    \begin{equation}  \label{eq:betaregein}
    |G( (m-i,m+i))-G((m,m))| < \beta^{-1}n^{1/3} + \beta i^2/n, \quad \forall i\in\Z.
    \end{equation}
    Here $G$ is the Busemann function from Section \ref{ssec:busemann}.
    The reasoning behind the design of these bounds is similar to those for \eqref{eq:betareg} and \eqref{eq:betarege}.
    Then $\ocE_2$ and $\cE_2$ are measurable with respect to $\{\omega_p: d(p)\ge 2m\}$, the random field on and above $\LL_{m}$.
    \item Let $\cE_3$ be the event where $\Gamma_\boo \cap \LL_{m_-+1} = \Gamma_\boo'\cap \LL_{m_-+1}$.
\end{enumerate}
Note that the definition of $\ocE_1$ involves not the geodesic $\Gamma_\boo$ but the geodesics $\cup_{|i|\le m_-} \Gamma_{\boo, (m_--i,m_-+i)}$ and similarly for $\ocE_2$. This helps make $\ocE_1,\cE_1$  independent of $\cE_2, \ocE_2$.
The probability \eqref{eq:p1p2e3} can now be written as $\PP[\{p_1, p_2 \in \Gamma_\boo\}\cap \cE_3^c]$, which is bounded by $\PP[\ocE_1\cap \ocE_2 \cap \cE_3^c]$.
The additional regularity conditions are added to allow us to invoke a crucial geodesic stability result that we will prove shortly (Proposition \ref{prop:stb} below) to bound the probability of $\cE_3^c$ (see Claim \ref{cla:w-be} below).

We assume the following claims when $n$ is large enough (depending on $w,\beta, \theta$).
\begin{cla}  \label{cla:total}
We have $\PP[\ocE_1], \PP[\ocE_2] < C_\theta n^{-2/3}$, where $C_\theta>0$ is a constant depending only on $\theta$.
\end{cla}

\begin{cla}  \label{cla:be-te}
For each $0<\beta, \theta<1$, there exists a constant $c_{\beta, \theta}^{bad}$, such that $\PP[\ocE_1\setminus \cE_1], \PP[\ocE_2\setminus \cE_2] < c_{\beta, \theta}^{bad}n^{-2/3}$ for $n$ large enough (depending on $\beta, \theta$), and $\lim_{\beta\to 0} c_{\beta, \theta}^{bad} = 0$ for fixed $\theta>0$.
\end{cla}

\begin{cla}  \label{cla:w-be}
For each $0<w, \beta, \theta<1$, there exists a constant $c_{w, \beta, \theta}^{disj}$, such that $\PP[\cE_3^c \mid \cE_1 \cap \cE_2] < c_{w, \beta, \theta}^{disj}$, and $\lim_{\beta\to 0}\limsup_{w\to 0} c_{w, \beta, \theta}^{disj} = 0$ for fixed $\theta>0$.
\end{cla}
Thus we must have that
\begin{align*}
\PP[p_1, p_2 \in \Gamma_\boo, p_1 \not\in \Gamma_\boo']
\le & \PP[\{p_1, p_2 \in \Gamma_\boo\}\cap \cE_3^c]
\\
\le & \PP[\ocE_1\cap \ocE_2 \cap \cE_3^c]
\\ 
\le & \PP[\ocE_1\setminus \cE_1] \PP[\ocE_2] +  \PP[\ocE_1]\PP[\ocE_2\setminus \cE_2] + \PP[\cE_1]\PP[\cE_2]\PP[\cE_3^c \mid \cE_1 \cap \cE_2]
\\
<& (2C_\theta c_{\beta, \theta}^{bad} + C_\theta^2 c_{w, \beta, \theta}^{disj}) n^{-4/3}.
\end{align*}
Let $c_{w,\theta}$ (in the statement of Proposition \ref{prop:occup-diff}) be $\inf_{\beta>0} 2C_\theta c_{\beta, \theta}^{bad} + C_\theta^2 c_{w, \beta, \theta}^{disj}$.
For fixed $\theta>0$, we have
\[\limsup_{w\to 0} c_{w,\theta} \le 2C_\theta c_{\beta, \theta}^{bad} + \limsup_{w\to 0} C_\theta^2 c_{w, \beta, \theta}^{disj}\]
for any $\beta>0$ in the right-hand side. Then by sending $\beta\to 0$ in the right-hand side, by Claims \ref{cla:be-te} and \ref{cla:w-be} we have $c_{w,\theta} \to 0$ as $w\to 0$.
Thus we get the desired upper bound for $\PP[p_1, p_2 \in \Gamma_\boo, p_1 \not\in \Gamma_\boo']$.
By a similar argument we can bound $\PP[p_1, p_2 \in \Gamma_\boo, p_3 \not\in \Gamma_\boo']$, thus we can upper bound $\PP[p_1, p_2 \in \Gamma_\boo] - \PP[p_1, p_3 \in \Gamma_\boo]$.
By symmetry we can also upper bound $\PP[p_1, p_3 \in \Gamma_\boo] - \PP[p_1, p_2 \in \Gamma_\boo]$, thereby finish the proof. \end{proof}

We now prove the claims.
For Claim \ref{cla:total} which bounds $\PP[\ocE_1]$ and $\PP[\ocE_2]$, we use translation invariance of Exponential LPP, and the bound on the number of disjoint geodesics (Lemma \ref{lem:num-int-dis}).
\begin{proof} [Proof of Claim \ref{cla:total}]
For each $j\in\Z$,  we let $\ocE_{1,j}$ be the event $\ocE_{1}$ translated by $(-j, j)$.
Then we have
\[
\PP[\ocE_1] = (2\lfloor n^{2/3}\rfloor +1)^{-1} \sum_{|j|\le n^{2/3}} \PP[\ocE_{1,j}] \le (2\lfloor n^{2/3}\rfloor +1)^{-1} \E[X],
\]
where \[X= |\{ p\in \Z^2: d(p)=d(p_1), |ad(p)| \le 3n^{2/3}, p \in \cup_{|i|\le n^{2/3}, j \in \Z} \Gamma_{(-i,i), (m_--j,m_-+j)} \}|.\] By Lemma \ref{lem:num-int-dis}, $\E[X]$ is bounded by a constant depending only on $\theta$. Thus we get the desired bound for $\PP[\ocE_1]$. 

Similarly, we have
\[
\PP[\ocE_2] \le (2\lfloor n^{2/3}\rfloor +1)^{-1} \E[Y],
\]
where \[Y= |\{ p\in \Z^2: d(p)=d(p_2), |ad(p)| \le 3n^{2/3}, p \in \cup_{j \in \Z} \Gamma_{(m-j,m+j)} \}|.\]
By Lemma \ref{lem:num-int-dis-inf}, $\E[Y]$ is bounded by a constant depending only on $\theta$ which then implies the sought bound for $\PP[\ocE_2]$. 
\end{proof}

We next prove Claim \ref{cla:be-te}.
In addition to the translation invariance arguments, to prove regularity and bound $\PP[\ocE_2\setminus \cE_2]$, we use the random walk description of the Busemann function on $\LL_m$.
For $\PP[\ocE_1\setminus \cE_1]$, we instead send $n\to\infty$ and use regularities of the directed landscape, because of the absence of exact random walk descriptions of the joint distribution of passage times for point-to-point geodesics in the pre-limit.
\begin{proof}[Proof of Claim \ref{cla:be-te}]
For each $j\in\Z$,  recall the events $\ocE_{1,j}$, $\ocE_{2,j}$ from the proof of Claim \ref{cla:total}. We also let $\cE_{1,j}$ and $\cE_{2,j}$ be the event $\cE_{1}$ and $\cE_{2}$ translated by $(-j, j)$, respectively.
Also recall the random variables $X$ and $Y$ from the proof of Claim \ref{cla:total}.

By translation invariance we have
\begin{equation}  \label{eq:ocEppX}
\PP[\ocE_1\setminus \cE_1] = (2\lfloor n^{2/3}\rfloor +1)^{-1} \sum_{|j|\le n^{2/3}} \PP[\ocE_{1,j}\setminus \cE_{1,j}] \le (2\lfloor n^{2/3}\rfloor +1)^{-1} \E[X\don[\cE_1'\cup \cE_1'']],    
\end{equation}
where $\cE_1'$ is the event where there exist $i, j, k \in \Z$, such that $|i|, |j|<2\beta^{-1}n^{2/3}$, $|k|\le n^{2/3}$, $wn^{2/3}<|i-j|<n^{2/3}$, and $|T_{(-k,k), (m_--i,m_-+i)} - T_{(-k,k), (m_--j,m_-+j)}| \ge \beta^{-2}|i-j|^{1/2}(n^{2/3}/|i-j|)^{0.01}$;
and $\cE_1''$ is the event where there exist $i, k \in \Z$, such that $|k|\le n^{2/3}$, and {$|T_{(-k,k), (m_--i,m_-+i)}-\E[T_{(-k,k), (m_--i,m_-+i)}]| \ge \beta^{-1}n^{1/3} + \beta (i-k)^2/n$.}
Namely, $\cE_1'$ (resp. $\cE_1''$) is the ``translation union version'' (in an order $n^{2/3}$ window) of the event that the regularity condition \eqref{eq:betareg} (resp. \eqref{eq:betarege}) does not hold.
    
By Theorem \ref{thm:exp-to-dl}, for fixed $\beta, \theta, \kappa_-$ (recall that $m_- = \lfloor \kappa_-\theta n/100 \rfloor$), we have that $\PP[\cE_1']$ converges as $n\to\infty$. 
The limit is the probability of the following event: there exist $x, y\in [-2^{1/3}\beta^{-1}, 2^{1/3}\beta^{-1}]$ with $w2^{-2/3}<|x-y|<2^{-2/3}$, and $z\in [-2^{-2/3}, 2^{-2/3}]$, such that 
\[|\cL(z,0;x,\kappa_-\theta/100) - \cL(z,0;y,\kappa_-\theta/100)|\ge 2^{-2}\beta^{-2}|x-y|^{1/2-0.01}.\]
Note that by the quadrangle inequality (Lemma \ref{lem:DL-quad}), if the above inequality holds, it must also hold when $z$ is replaced by either $-2^{-2/3}$ or $2^{-2/3}$.
Then by covering $[-2^{1/3}\beta^{-1}, 2^{1/3}\beta^{-1}]$ with overlapping intervals, we have that the above probability is bounded by the probability of the following event: there exist $i\in\Z$, $|i|<2\beta^{-1}$, and $x,y \in [(i-1)2^{-2/3}, (i+1)2^{-2/3}]$ such that \[|\cL(2^{-2/3},0;x,\kappa_-\theta/100) - \cL(2^{-2/3},0;y,\kappa_-\theta/100)|\ge 2^{-2}\beta^{-2}|x-y|^{1/2-0.01},\] or \[|\cL(-2^{-2/3},0;x,\kappa_-\theta/100) - \cL(-2^{-2/3},0;y,\kappa_-\theta/100)|\ge 2^{-2}\beta^{-2}|x-y|^{1/2-0.01}.\]
Since $x\mapsto (\kappa_-\theta/100)^{-1/3}\cL(2^{-2/3},0;x+2^{-2/3},\kappa_-\theta/100)$ is the Airy$_2$ process, we can bound this probability for each $i$, using that the Airy$_2$ process plus a parabola is stationary and Theorem \ref{thm:airytail}, and take a union bound for all $i$. {We conclude that (for fixed $\beta, \theta$)
$\lim_{n\to\infty} \PP[\cE_1'] < C\beta^{-1} e^{-c\beta^{-2}}$,
where $C,c>0$ are constants depending on $\theta, \kappa_-$. }

For $\ocE_1''$,
we use estimates in Section \ref{sssec:ptest} to bound its probability.
A key thing to notice is the slope conditions in these estimates.
When $|i|\le m_-/2$, the slope of $(m_--i+k,m_-+i-k)$ is always bounded away from $0$ and $\infty$.
Thus we split $\{(m_- - i, m_- + i): |i|\le m_-/2\}$ and $\{(-k, k): |k|\le n^{2/3}\}$ into segments of length $m_-^{2/3}$, and apply Proposition \ref{prop:seg-to-seg} to each one of them and take a union bound.
When $|i|> m_-/2$, the slope can be close to $0$ or very large, in which cases Proposition \ref{prop:seg-to-seg} cannot be applied,
so we just apply \eqref{e:wslope} in Theorem \ref{thm:lpp-onepoint} to each $T_{(-k,k),(m_--i,m_-+i)}$ and take a union bound.
In the end, we conclude that (for fixed $\beta, \theta$) $\limsup_{n\to\infty}\PP[\cE_1'']<C\beta^{-1/2}e^{-c\beta^{-1}}$, where $C,c>0$ are constants depending on $\theta, \kappa_-$.

By considering all $\kappa_- \in \llbracket 1, 100\theta^{-1}\rrbracket$, we have $\PP[\cE_1'\cup\cE_1'']<C\beta^{-1/2}e^{-c\beta^{-1}}$ for constants $C,c>0$ depending on $\theta$.
Besides, by Lemma \ref{lem:num-int-dis} we also have that $\PP[X>M]<Ce^{-cM^{1/12800}}$ for any $M>0$, with $C,c>0$ being constants that depend on $\theta$.
{Thus by \eqref{eq:ocEppX}, we get the desired bound for $\PP[\ocE_1\setminus \cE_1]$.}

For $\PP[\ocE_2\setminus\cE_2]$, we similarly have 
\[
\PP[\ocE_2\setminus\cE_2] \le (2\lfloor n^{2/3}\rfloor +1)^{-1} \E[Y\don[\cE_2'\cap \cE_2'']],
\]
where $\cE_2'$ is the event where there exist $i, j\in \Z$, such that $|i|, |j|<2\beta^{-1}n^{2/3}$, $wn^{2/3}<|i-j|<n^{2/3}$,
and $|G((m-i,m+i)) - G((m-j,m+j))| \ge \beta^{-2}|i-j|^{1/2}(n^{2/3}/|i-j|)^{0.01}$; and $\cE_2''$ is the event where there exists $i \in \Z$, such that $|G( (m-i,m+i))-G((m,m))| \ge \beta^{-1}n^{1/3} + \beta i^2/n$.
We note that $j\mapsto G((m-j,m+j))-G((m,m))$ is a random walk where each step has Exponential tails (see the second property in Section \ref{ssec:busemann}).
Then we have that for fixed $\beta, \theta, \kappa$, the probabilities $\PP[\cE_2']$ and $\PP[\cE_2'']$ converge as $n\to\infty$; and for the limits, they decay to zero as $\beta \to 0$, for any fixed $\theta, \kappa$.
Thus with the Exponential tail of $Y$ obtained from Lemma \ref{lem:num-int-dis-inf}, we get the desired bound for $\PP[\ocE_2\setminus \cE_2]$.
\end{proof}

Finally, for Claim \ref{cla:w-be}, we use the following stability result on geodesics.

Let $H_1, H_2>0$.
For each $w > 0$ and $n\in\N$, we let $\sF_{w,n}$ be the collection of all functions $f:\llbracket -H_1n^{2/3}, H_1n^{2/3}\rrbracket \to \R$, such that $|f|\le H_2n^{1/3}$,
and $|f(i)-f(j)| < H_2|i-j|^{1/2}(n^{2/3}/|i-j|)^{0.01}$ for any $|i|, |j|\le H_1n^{2/3}$, $wn^{2/3}<|i-j|<n^{2/3}$.
For each $f_1, f_2 \in \sF_{w,n}$, suppose
\[
i_*, j_* = \argmax_{i,j\in \llbracket -H_1n^{2/3}, H_1n^{2/3}\rrbracket} f_1(i) + f_2(j) + T_{(-i,i),(n-j,n+j)},
\]
then we let $\Gamma_{f_1,f_2}^n=\Gamma_{(-i_*,i_*),(n-j_*,n+j_*)}$.
Note that almost surely, such $i_*$ and $j_*$ are unique, because the passage times $T_{(-i,i),(n-j,n+j)}$ have continuous distributions. The next result states that $\Gamma_{f_1,f_2}^n$ overlaps significantly with $\Gamma_{f_1,f_3}^n$
\begin{prop}\label{prop:stb}
For each $\epsilon>0$, the following is true for any small enough $w, g>0$ (depending on $H_1, H_2$) and any $n\in\N$ large enough (depending on $w, g$ and $H_1, H_2$).
Take any $f_1, f_2, f_3 \in \sF_{w,n}$ with $|f_2-f_3|\le gn^{1/3}$, we have
\[
\PP[\Gamma_{f_1,f_2}^n \cap \LL_{\lfloor n/2\rfloor} = \Gamma_{f_1,f_3}^n \cap \LL_{\lfloor n/2\rfloor}] > 1-\epsilon.
\]
\end{prop}
We leave its proof to the next section and prove Claim \ref{cla:w-be} now.

\begin{proof} [Proof of Claim \ref{cla:w-be}]
Denote $l=m-m_--1$. Let $H_1>0$ be a constant depending only on $\beta, \theta$, such that we always have $\beta^{-0.9}n^{2/3} < H_1l^{2/3} < \beta^{-1}n^{2/3}$.
Such an $H_1$ exists whenever $\beta$ is small enough (depending on $\theta$).
Let $H_2>0$ be large enough depending on $\beta, \theta$.

We want to apply Proposition \ref{prop:stb} to the strip between $\LL_{m_-}$ and $\LL_m$. The regularity conditions (of Proposition \ref{prop:stb}) are given by $\cE_1\cap \cE_2$. 
Also, Proposition \ref{prop:stb} considers boundary functions in $\sF_{w,n}$, which are defined on the finite interval $\llbracket -H_1n^{2/3}, H_1n^{2/3}\rrbracket$.
To apply it we need to bound transversal fluctuations of $\Gamma_\boo$ and $\Gamma_\boo'$.

Towards this, we first express parts of $\Gamma_\boo$ and $\Gamma_\boo'$ in the strip between $\LL_{m_-}$ and $\LL_m$ in terms of boundary conditions.
We let $G'$ be the Busemann function for the random field $\{\omega_p': p\in\Z^2\}$.
Because of the shift by $p_3-p_2,$ we then have that $G(p)-G(q) = G'(p+p_3-p_2) - G'(q+p_3-p_2)$, for any $p, q\in\Z^2$ on or above $\LL_m$.
Next, define $f_1, f_2, f_3:\Z \to \R$ as
\[
f_1(i) = T_{\boo, (m_-+i,m_--i)} \vee T_{\boo, (m_-+i+1,m_--i-1)} - 4m_-,
\]
\[
f_2(i) = G( (m+i,m-i))\vee G( (m+i+1,m-i-1)) -G((m,m)),
\]
\[
f_3(i) = G'( (m+i,m-i))\vee G'( (m+i+1,m-i-1)) -G'((m,m)).
\]
In these definitions, we take the maximum over two nearest points, because for $\Gamma_\boo$ or $\Gamma_\boo'$, the immediate steps beyond the strip (between $\LL_{m_-}$ and $\LL_m$) can take one of two possible vertices.
We let $\Gamma_*=\Gamma_{(m_-+i_*+1,m_--i_*), (m+j_*, m-j_*-1)}$, where
\[
i_*, j_* = \argmax_{i,j\in \Z} f_1(i) + f_2(j) + T_{(m_-+i+1,m_--i),(m+j,m-j-1)};
\]
and define $\Gamma_*'=\Gamma_{(m_-+i_*'+1,m_--i_*'), (m+j_*', m-j_*'-1)}$ the same way as $\Gamma_*$, but using $f_3$ in place of $f_2$.
{We then have that $\Gamma_*$ is the part of $\Gamma_\boo$ between $\LL_{m_-}$ and $\LL_m$, and similarly $\Gamma_*'$ is the part of $\Gamma_\boo'$ between $\LL_{m_-}$ and $\LL_m$.}

For the regularity conditions, we need (1) the functions $f_1, f_2, f_3$ (restricted to finite intervals) to be in $\sF_{w,n}$, and (2) $f_2$ to be close to $f_3$.
Specifically, for (1), we let $\cE_{reg}$ be the event where for each $\iota = 1,2,3$, $f_\iota$ on $\llbracket -H_1l^{2/3}, H_1l^{2/3}\rrbracket$ is in $\sF_{w,n}$ (i.e., $|f_\iota| < H_2 l^{1/3}$ on $\llbracket -H_1l^{2/3}, H_1l^{2/3}\rrbracket$; and $|f_\iota(i)-f_\iota(j)| < H_2 |i-j|^{1/2}(l^{2/3}/|i-j|)^{0.01}$ for any $|i|, |j| \le H_1l^{2/3}$, $wl^{2/3} < |i-j| < l^{2/3}$).
For (2), we let $\cE_{diff}$ be the event where $|f_2-f_3|\le H_2w^{1/2-0.01}l^{1/3}$ on $\llbracket -H_1l^{2/3}, H_1l^{2/3}\rrbracket$.
We have that these events are essentially implied by $\cE_1\cap \cE_2$.
Indeed, when $d(p_2)=d(p_3)$, one precisely has $\cE_1\cap \cE_2\subset \cE_{reg}\cap \cE_{diff}$ (since $H_1l^{2/3} < \beta^{-1}n^{2/3}$ by the choice of $H_1$).
When $d(p_3)=d(p_2)+1$, $\cE_1\cap \cE_2$ does not imply $\cE_{reg}\cap \cE_{diff}$, because of the extra random weights $\{\omega_{(m+i,m-i)}': |i|\le H_1l^{2/3}\}$ (which are independent of $\cE_1\cap\cE_2$).
However, as long as each $\omega_{(m+i,m-i)}'$ is not too large (i.e., $\le w^{1/2-0.01}l^{1/3}$), the event $\cE_{reg}\cap \cE_{diff}$ still holds.
Therefore we have
\[\PP[\cE_{reg}^c\cup \cE_{diff}^c\mid \cE_1\cap \cE_2] < \PP[\max_{|i|\le H_1l^{2/3}} \omega_{(m+i,m-i)}' > w^{1/2-0.01}l^{1/3}],\]
and the right hand side $\to 0$ as $n\to\infty$, for fixed $w,\beta,\theta$.
Then we have
\begin{equation}  \label{eq:clapf2}
\lim_{n\to\infty} \PP[\cE_{reg}^c\cup \cE_{diff}^c\mid \cE_1\cap \cE_2] =0   
\end{equation}
for any $w, \beta, \theta$.

To apply Proposition \ref{prop:stb}, 
we need that the optimal paths with endpoints in $\{(m_-+i+1,m_--i): i\in \llbracket -H_1l^{2/3}, H_1l^{2/3}\rrbracket\}$ and $\{(m+j+1,m-j): j\in \llbracket -H_1l^{2/3}, H_1l^{2/3}\rrbracket\}$ are the same as the overall optimal paths $\Gamma_*$ and $\Gamma_*'$.
For this, we need to bound $\PP[\cE_{fluc}\mid \cE_1\cap \cE_2]$, for $\cE_{fluc}$ being the event that $|i_*|\vee |i_*'|\vee |j_*|\vee |j_*'| > H_1l^{2/3}$.
For this, we show that when $n$ is large enough, given any $\{\omega_p: d(p)\le 2m_-\}$ and $\{\omega_p: d(p)\le 2m_-\}$ such that $\cE_1\cap\cE_2$ holds, the probability of the event where
\begin{multline} \label{eq:ijtransla}
\max_{i,j\in \Z, |i|\vee |j|>H_1l^{2/3}} f_1(i) + f_2(j) + T_{(m_-+i+1,m_--i),(m+j,m-j-1)} \\ > f_1(0) + f_2(0) + T_{(m_-+1,m_-),(m,m-1)}
\end{multline}
or
\begin{multline} \label{eq:ijtranslat}
\max_{i,j\in \Z, |i|\vee |j|>H_1l^{2/3}} f_1(i) + f_3(j) + T_{(m_-+i+1,m_--i),(m+j,m-j-1)} \\ > f_1(0) + f_3(0) + T_{(m_-+1,m_-),(m,m-1)}
\end{multline}
is bounded by $c^{div}_{\beta,\theta}$, which depends only on $\beta, \theta$ and satisfies that $\lim_{\beta\to 0} c^{div}_{\beta,\theta} = 0$.
By \eqref{eq:betarege} and \eqref{eq:betaregein} from the event $\cE_1\cap\cE_2$, and the expectation estimate \eqref{e:mean}, we have that for any $i, j \in \Z$,
\begin{multline*}
(f_1(i)+ i^2/m_--\beta^{-1}n^{1/3}-\beta i^2/n) + f_2(j) + (\E[T_{(m_-+i+1,m_--i),(m+j,m-j-1)}] + (i-j)^2/l)\\
< (f_1(0)+\beta^{-1}n^{1/3}+\beta i^2/n) + (f_2(0)+\beta^{-1}n^{1/3}+\beta j^2/n) + \E[T_{(m_-+1,m_-),(m,m-1)}] + Cn^{1/3},
\end{multline*}
where $C>0$ is a universal constant.
Recall from the beginning of this proof that $H_1$ is chosen such that $H_1l^{2/3} > \beta^{-0.9}n^{2/3}$.
Thus when $|i|\vee |j|>H_1l^{2/3}>\beta^{-0.9}n^{2/3}$, we must have
\begin{multline*}
f_1(i) + f_2(j) + \E[T_{(m_-+i+1,m_--i),(m+j,m-j-1)}] \\ < f_1(0) + f_2(0) + \E[T_{(m_-+1,m_-),(m,m-1)}] - c ((i-j)^2+\beta^{-1.8})/n,    
\end{multline*}
where $c>0$ is another universal constant.
Then we can bound the probability of \eqref{eq:ijtransla} conditioned on $\cE_1\cap\cE_2$, by using Theorem \ref{thm:lpp-onepoint} and Proposition \ref{prop:seg-to-seg} and spitting $\Z$ into intervals of length $l^{2/3}$.
The probability of \eqref{eq:ijtranslat} conditioned on $\cE_1\cap\cE_2$ can be bounded similarly.
We then conclude that
\begin{equation}  \label{eq:clapf1}
\PP[\cE_{fluc}\mid \cE_1\cap \cE_2] < c^{div}_{\beta,\theta}   
\end{equation}
when $n$ is large enough.

Unless $\cE_{fluc}$ happens, the optimal paths $\Gamma_*$ and $\Gamma_*'$ are the same as the optimal paths with endpoints restricted to two intervals of length $2H_1l^{2/3}$.
Therefore, given any $f_1, f_2, f_3$ such that $\cE_{reg}\cap \cE_{diff}$ holds, Proposition \ref{prop:stb} applies. It implies that with high probability, either $\Gamma_*\cap \LL_{m_-+1} = \Gamma_*'\cap \LL_{m_-+1}$ (which is equivalent to $\cE_3$) or $\cE_{fluc}$ happens.
More precisely, we have
\begin{equation}  \label{eq:clapf3}
\lim_{w\to 0}\liminf_{n\to\infty}\PP[\cE_3\cup \cE_{fluc} \mid \cE_{reg}\cap \cE_{diff} \cap \cE_1 \cap \cE_2]
> c^{sta}_{w,\beta,\theta},  
\end{equation}
when $n$ is large enough; here $c^{sta}_{w,\beta,\theta}$ depends on $w, \beta, \theta$ and satisfies that $\lim_{w\to 0} c^{sta}_{w,\beta,\theta}= 1$.

By the elementary inequality of $x+(1-y)\ge (x\wedge y)/y$ for any $x,y\in (0,1)$,
we have
\[
\PP[\cE_3\cup \cE_{fluc} \mid \cE_{1}\cap \cE_{2} ] + \PP[\cE_{reg}^c\cup \cE_{diff}^c\mid \cE_1\cap \cE_2]
\ge 
\PP[\cE_3\cup \cE_{fluc} \mid \cE_{reg}\cap \cE_{diff} \cap \cE_1 \cap \cE_2].
\]
Thus we have
\begin{multline*}
\PP[\cE_3\mid \cE_1\cap \cE_2] \ge 
\PP[\cE_3\cup \cE_{fluc} \mid \cE_{reg}\cap \cE_{diff} \cap \cE_1 \cap \cE_2]
\\
- \PP[\cE_{reg}^c\cup \cE_{diff}^c\mid \cE_1\cap \cE_2] - \PP[\cE_{fluc}\mid \cE_1\cap \cE_2]. 
\end{multline*}
Then by plugging \eqref{eq:clapf2}, \eqref{eq:clapf1}, \eqref{eq:clapf3} to the right-hand side, we have that the conclusion holds.
\end{proof}

\section{Stability with respect to boundary conditions}  \label{ssec:stab}

In this section, we prove Proposition \ref{prop:stb}. 
Throughout this section, we take $H_1, H_2 > 0$ as two fixed constants, and all other constants are allowed to depend on them.

As already indicated in Section \ref{sec:iop}, the general idea to prove Proposition \ref{prop:stb} involves two steps. In the first, we show that with high probability, the geodesics $\Gamma_{f_1,f_2}^n$ and $\Gamma_{f_1,f_3}^n$ share a significant part of their journey being quite close to each other; in the second step we show that the conclusion of the first step implies that the geodesics coalesce with high probability. The two steps are carried out in  Sections \ref{ssec:stabw} and \ref{ssec:close-to-coal} respectively.

\subsection{Closeness of geodesics at one time} \label{ssec:stabw}
\begin{prop}\label{prop:stb-weak}
For each $\epsilon, \lambda>0$, $0.1<\alpha <0.9$, the following is true for any small enough $w, g>0$ and any $n\in\N$ large enough (depending on $w, g$).
Take any $f_1, f_2, f_3 \in \sF_{w,n}$ with $|f_2-f_3|\le g n^{1/3}$, we have
\[
\PP[|\Gamma_{f_1,f_2}^n \cap \LL_{\lfloor \alpha n\rfloor} - \Gamma_{f_1,f_3}^n \cap \LL_{\lfloor \alpha n\rfloor}| < \lambda n^{2/3}] > 1-\epsilon.
\]
\end{prop}

To prove this, we first prove a limiting version in the directed landscape. While this is in principle not necessary, taking advantage of the local Brownian property (Theorem \ref{thm:abs-bm-gen}) in the limit is convenient for the exposition and helps make the proof cleaner. 

Overall the underlying observation is that the non-closeness of the geodesic implies a ``twin-peaks'' event, i.e., where certain passage time profiles on $\LL_{\lfloor \alpha n\rfloor}$ are nearly maximized at two points with distance $\ge \lambda n^{2/3}$. In the directed landscape, this is essentially the same as a Brownian motion having a ``near-maxima'' significantly away from the global maximum, an event whose probability is small.
The limiting twin-peaks event estimate is stated as follows.
We let $\sF$ be the collection of all upper semi-continuous functions $f:[-2^{-2/3}H_1, 2^{-2/3}H_1] \to \R$, such that $|f|\le 2^{-4/3}H_2$.
For each $f\in \sF$ we denote its hypograph $\hypo(f)$ as
\[
\hypo(f) = \{(x,y): x \in [-2^{-2/3}H_1, 2^{-2/3}H_1], -2^{-4/3}H_2 \le y \le f(x)\},
\]
which is a closed subset of $[-2^{-2/3}H_1, 2^{-2/3}H_1] \times [-2^{-4/3}H_2, 2^{-4/3}H_2]$.
Since the space of all compact subsets of $[-2^{-2/3}H_1, 2^{-2/3}H_1] \times [-2^{-4/3}H_2, 2^{-4/3}H_2]$ is complete (with the Hausdorff distance metric), for any sequence of functions in $\sF$, we can find a subsequence $\{f_k\}_{k\in\N}$, such that $\hypo(f_k)$ converges in the Hausdorff topology as $k\to\infty$.
It is straightforward to check that the limit must be $\hypo(f)$ for some function $f\in \sF$.

Below we let $\cH(\cdot, \cdot)$ denote the Hausdorff distance between compact sets in $\R^2$.
\begin{lemma}  \label{lem:stb-conti}
For each $\epsilon, \lambda>0$, $0.1<\alpha <0.9$, and $H'>0$, there is $\delta>0$, such that the following is true.

Take any $f_1, f_2\in \sF$, 
define $f_*:\R\to\R$ as \[f_*(x) = \sup_{|y|,|z|\le 2^{-2/3}H_1}\cL(y,0;x,\alpha) + \cL(x,\alpha;z,1) + f_1(y) + f_2(z),\]
and we let $\NTP_\delta(f_1, f_2)$ denote the following event: for any $x_1,x_2\in [-2^{-2/3}H', 2^{-2/3}H']$ with $x_2-x_1\ge 2^{-2/3}\lambda$ we have $f_*(x_1)\wedge f_*(x_2) \le \sup_{x\in [-2^{-2/3}H',2^{-2/3}H']} f_*(x) - 2^{-4/3}\delta$ (i.e., the event of no-twin-peaks).
Then we have $\PP[\NTP_\delta(f_1, f_2)]>1-\epsilon$.
\end{lemma}
\begin{proof}
It suffices to show that
$\lim_{\delta\to 0} \inf_{f_1,f_2\in \sF} \PP[\NTP_\delta(f_1, f_2)] = 1$.
We argue by contradiction, and suppose that for each $k\in\N$ we can find some $\delta_k>0$, and $f_{1,k}, f_{2,k} \in \sF$, such that $\lim_{k\to\infty}\delta_k = 0$, and $\lim_{k\to\infty} \PP[\NTP_{\delta_k}(f_{1,k}, f_{2,k})] = \gamma < 1$.
By taking a subsequence we can assume that $\hypo(f_{1,k})\to \hypo(f_{1,\infty})$ and $\hypo(f_{2,k})\to \hypo(f_{2,\infty})$ in the Hausdorff topology, for some $f_{1,\infty}, f_{2,\infty} \in \sF$.
In other words, if we let $\tau_k=\cH(\hypo(f_{1,k}), \hypo(f_{1,\infty})) \vee \cH(\hypo(f_{2,k}), \hypo(f_{2,\infty}))$, then $\tau_k\to 0$ as $k\to\infty$.

We now explain how we will reach a contradiction.
Define $f_{*,k}:\R\to\R$ by \[f_{*,k}(x) = \sup_{|y|,|z|\le 2^{-2/3}H_1}\cL(y,0;x,\alpha) + \cL(x,\alpha;z,1) + f_{1,k}(y) + f_{2,k}(z),\] for $k\in \N\cup\{\infty\}$.
On one hand, $f_{*,\infty}$ is local Brownian (by Theorem \ref{thm:abs-bm-gen}), so almost surely it has a unique maximum in any compact interval.
However, we will use $\lim_{k\to\infty} \PP[\NTP_{\delta_k}(f_{1,k}, f_{2,k})] = \gamma < 1$ to deduce that with probability $\ge 1-\gamma$, the maximum of $f_{*,\infty}$ in $[-2^{-2/3}H', 2^{-2/3}H']$ is not unique.
This can be achieved by bounding the difference between $f_{*,k}$ and $f_{*,\infty}$, using Lemma \ref{lem:modcont}.

We next fill in the details of this outlined strategy.
Take $R$ as given by Lemma \ref{lem:modcont}, where the compact set $K$ is taken as the union of \[([-2^{-2/3}H_1,2^{-2/3}H_1]\times [0,\alpha/3]) \times ([-2^{-2/3}H',2^{-2/3}H']\times [2\alpha/3, 2\alpha/3+1/3])\] and \[([-2^{-2/3}H',2^{-2/3}H']\times [2\alpha/3, 2\alpha/3+1/3]) \times ([-2^{-2/3}H_1,2^{-2/3}H_1]\times [\alpha/3+2/3, 1]).\]
We take $K$ this way so that we can now bound $|\cL(y,0;x,\alpha)-\cL(y',0;x,\alpha)|$ and $|\cL(x,\alpha;z,1)-\cL(x,\alpha;z',1)|$, for any $|x|\le 2^{-2/3}H'$ and $|y|,|z|, |y'|, |z'|\le 2^{-2/3}H_1$.

Using $\tau_k$ defined above, for any $|x|\le 2^{-2/3}H'$ and $|y|,|z|\le 2^{-2/3}H_1$, we can find $y', z'$ such that $|y'|,|z'|\le 2^{-2/3}H_1$, $|y-y'|, |z-z'|\le \tau_k$, and $f_{1,k}(y') \ge f_{1,\infty}(y)-\tau_k$, $f_{2,k}(z') \ge f_{2,\infty}(z)-\tau_k$.
Thus
\[f_{*,k}(x)- f_{*,\infty}(x) \ge - (C\tau_k + 2R\tau_k^{1/2}\log^{1/2}(1+\tau_k^{-1/2})),\]
where $C>0$ is a constant depending on $H'$ and $H_1$.
Similarly, the same lower bound holds for $f_{*,\infty}(x)- f_{*,k}(x)$.
So we have
\[
|f_{*,k}(x) - f_{*,\infty}(x)| \le C\tau_k + 2R\tau_k^{1/2}\log^{1/2}(1+\tau_k^{-1/2}).
\]
Given that $R\le \tau_k^{-1/4}\log^{-1/2}(1+\tau_k^{-1/2})$, the right-hand side is $\le 2\tau_k^{1/4}+C\tau_k$, and we have that $\NTP_{\delta_k+2\tau_k^{1/4}+C\tau_k}(f_{1,\infty}, f_{2,\infty})$ implies $\NTP_{\delta_k}(f_{1,k}, f_{2,k})$.
Therefore we have
\[
\PP[\NTP_{\delta_k+2\tau_k^{1/4}+C\tau_k}(f_{1,\infty}, f_{2,\infty})] \le \PP[\NTP_{\delta_k}(f_{1,k}, f_{2,k})] + \PP[R>\tau_k^{-1/4}\log^{-1/2}(1+\tau_k^{-1/2})].
\]
Thus $\limsup_{k\to\infty} \PP[\NTP_{\delta_k+2\tau_k^{1/4}+C\tau_k}(f_{1,\infty}, f_{2,\infty})] \le \gamma$.

If $f_{*,\infty}$ has a unique maximum in $[-2^{-2/3}H', 2^{-2/3}H']$, for any sufficiently large $k$ the event $\NTP_{\delta_k+2\tau_k^{1/4}+C\tau_k}(f_{1,\infty}, f_{2,\infty})$ holds.
So the probability of $f_{*,\infty}$ having a unique maximum in $[-2^{-2/3}H', 2^{-2/3}H']$ is at most $\gamma$.

By Theorem \ref{thm:abs-bm-gen} the functions
\[
x\mapsto \sup_{|y|\le 2^{-2/3}H_1}\cL(y,0;x,\alpha) + f_{1,k}(y), \quad x\mapsto \sup_{|z|\le 2^{-2/3}H_1}\cL(x,\alpha;z,1) + f_{2,k}(z)
\]
are absolutely continuous with respect to Brownian motions in any compact interval. (Note that the suprema are taken over compact intervals.)
Then we have that their sum $f_{*,\infty}$ in $[-2^{-2/3}H', 2^{-2/3}H']$ is absolutely continuous with respect to a Brownian motion, thus has a unique maximum almost surely. 
Thus we get a contradiction.
\end{proof}
We next use this estimate on the twin-peaks event in the limiting setting to prove Proposition \ref{prop:stb-weak}.
Note that the estimate in Proposition \ref{prop:stb-weak} is required to hold uniformly for any $f_1, f_2, f_3 \in \sF_{w,n}$.
To achieve this we will take a finite family of functions so that any function in $\sF_{w,n}$ is close to one in this family.

From now on, let $T_{p,q}^*=T_{p,q}-\omega_p$ for any $\Z^2$ vertices $p\preceq q$.
Removing the weight of the first vertex as such helps to get useful independence between passage times, as we will soon see.
\begin{proof}[Proof of Proposition \ref{prop:stb-weak}]
We first define the no-twin-peaks event in Exponential LPP.
Let $H'>0$ be a large number depending on $\epsilon, \alpha$, and the requirement for its value will be evident from the arguments below.
For each $f_1, f_2:\llbracket -H_1n^{2/3}, H_1n^{2/3}\rrbracket \to \R$, we let $T_{f_1,f_2}:\Z\to \R$ be the function where
\[
T_{f_1,f_2}(i) = \max_{|i_1|, |i_2|\le H_1n^{2/3}} T_{(i_1,-i_1),(\lfloor \alpha n\rfloor +i,\lfloor \alpha n\rfloor -i)} + T^*_{(\lfloor \alpha n\rfloor +i,\lfloor \alpha n\rfloor -i), (n+i_2,n-i_2)} + f_1(i_1)+f_2(i_2).
\]
Let
$\NTP_{\delta,n}(f_1, f_2)$ be the event where for any $i_1, i_2 \in \llbracket -H'n^{2/3}, H'n^{2/3}\rrbracket$ with $i_2-i_1 \ge \lambda n^{2/3}$, one has $T_{f_1,f_2}(i_1)\wedge T_{f_1,f_2}(i_2) \le \max_{|i|\le H'n^{2/3}}T_{f_1,f_2}(i)-\delta n^{1/3}$.

As the no-twin-peaks event only involves the function $T_{f_1,f_2}$ on the interval $\llbracket -H'n^{2/3}, H'n^{2/3}\rrbracket$, we need to bound the transversal fluctuation of geodesics.
By Lemma \ref{lem:trans-dis-fini}, by taking $H'>0$ large, for any $n$ large enough we have $\PP[\cE_{trans}] > 1-\epsilon/2$, where $\cE_{trans}$ is the event
\[
\Gamma_{(-\lfloor H_1n^{2/3}\rfloor, \lfloor H_1n^{2/3}\rfloor), (n-\lfloor H_1n^{2/3}\rfloor, n+\lfloor H_1n^{2/3}\rfloor)} \cap \LL_{\lfloor \alpha n \rfloor}\subset \{(\lfloor \alpha n \rfloor-i,\lfloor \alpha n \rfloor+i): |i|\le H'n^{2/3}\} ,
\]
and
\[
\Gamma_{(\lfloor H_1n^{2/3}\rfloor, -\lfloor H_1n^{2/3}\rfloor), (n+\lfloor H_1n^{2/3}\rfloor, n-\lfloor H_1n^{2/3}\rfloor)} \cap \LL_{\lfloor \alpha n \rfloor} \subset \{(\lfloor \alpha n \rfloor-i,\lfloor \alpha n \rfloor+i): |i|\le H'n^{2/3}\} .
\]
Then under $\cE_{trans}$, for any $f_1, f_2 \in \sF_{w,n}$ we must have that
\[
\Gamma_{f_1,f_2}^n \cap \LL_{\lfloor \alpha n\rfloor} \subset \{(\lfloor \alpha n \rfloor-i,\lfloor \alpha n \rfloor+i): |i|\le H'n^{2/3}\},
\]
since any two geodesics cannot cross each other twice.

Now by taking $w, g$ small and then $n$ large, we have that for any $f_1, f_2, f_3 \in \sF_{w,n}$ with $|f_2-f_3|\le g n^{1/3}$,
\begin{equation}  \label{eq:fotbd}
\begin{split}
& \PP[|\Gamma_{f_1,f_2}^n \cap \LL_{\lfloor \alpha n\rfloor} - \Gamma_{f_1,f_3}^n \cap \LL_{\lfloor \alpha n\rfloor}| < \lambda n^{2/3}] \\
\ge & \PP[\NTP_{g,n}(f_1, f_2)] - (1-\PP[\cE_{trans}]) \\
> & \PP[\NTP_{g,n}(f_1, f_2)] - \epsilon/2.
\end{split}    
\end{equation}
Here the first inequality is due to that, if the intersections $\Gamma_{f_1,f_2}^n \cap \LL_{\lfloor \alpha n\rfloor}$ and $\Gamma_{f_1,f_3}^n \cap \LL_{\lfloor \alpha n\rfloor}$ are $\ge \lambda n^{2/3}$ from each other, and the transversal fluctuation of geodesics is bounded (i.e., $\cE_{trans}$ happens), then the twin-peaks event (i.e., the complement of $\NTP_{g,n}(f_1, f_2)$) must occur.
Then it remains to bound the probability of the latter, for which we use the limiting estimate given by Lemma \ref{lem:stb-conti}.

For any $f:\llbracket -H_1n^{2/3}, H_1n^{2/3}\rrbracket \to \R$, let $\sfS_n f$ be the function on $[-2^{-2/3}H_1, 2^{-2/3}H_1]$, such that $\sfS_n f(x) = 2^{-4/3}n^{-1/3}f(\lfloor (2n)^{2/3}x \rfloor)$.
Note that this is well defined only when $\lfloor (2n)^{2/3}x \rfloor \ge -H_1n^{2/3}$. For $x$ near $-2^{-2/3}H_1$ (i.e., when $\lfloor (2n)^{2/3}x \rfloor < -H_1n^{2/3}$), we instead let $\sfS_n f(x) = 2^{-4/3}n^{-1/3}f(\lceil (2n)^{2/3}x \rceil)$.

To bound the twin-peaks probability uniformly in the boundary functions, we take a finite family of functions in $\sF_{w,n}$ to approximate every function in $\sF_{w,n}$, and do limit transition for each function in this finite family to use Lemma \ref{lem:stb-conti}.

Towards this, for any $w, g > 0$ and $n\in\N$, we can find $\{f_{w,g,n;k}\}_{k=1}^{m_{w,g}}\subset \sF_{w,n}$, and $\{f_{w,g;k}\}_{k=1}^{m_{w,g}} \subset \sF$, where $m_{w,g}\in\N$  depends only on $w, g$, such that
\begin{enumerate}
    \item $\sfS_n f_{w,g,n;k} \to f_{w,g;k}$ uniformly as $n\to\infty$;
    \item for each $f\in \sF_{w,n}$, there is $f_{w,g,n;k}$, such that $|f-f_{w,g,n;k}| \le gn^{1/3}$.
\end{enumerate}
One way to choose such functions is as follows. 
First, take some $\iota>0$ small enough (depending on $w,g$). 
Let $\{f_{w,g;k}\}_{k=1}^{m_{w,g}}$ contain all functions on $[-2^{-2/3}H_1, 2^{-2/3}H_1]$, which (1) take values in $\iota \Z\cap[-2^{-4/3}H_2, 2^{-4/3}H_2]$ at $(\iota\Z\cap [-2^{-2/3}H_1, 2^{-2/3}H_1])\cup\{-2^{-2/3}H_1, 2^{-2/3}H_1\}$, (2) are linear on the rest of the domain, and (3) are in the closure (in the uniform convergence topology) of $\cup_{n\in\N}\{\sfS_n f: f\in \sF_{w,n}\}$.
Then we let $f_{w,g,n;k}$ be the function in $\sF_{w,n}$ with the smallest $\|\sfS_n f_{w,g,n;k}-f_{w,g;k}\|_{\infty}$.

By Lemma \ref{lem:stb-conti}, we can take $\delta>0$ small enough, such that $\PP[\NTP_\delta(f_1,f_2)]>1-\epsilon/4$ for any $f_1, f_2\in \sF$.
By Theorem \ref{thm:exp-to-dl}, we have 
\begin{equation}  \label{eq:pcov}
\lim_{n\to\infty}\PP[\NTP_{\delta,n}(f_{w,g,n;k_1}, f_{w,g,n;k_2})]= \PP[\NTP_\delta(f_{w,g;k_1}, f_{w,g,k_2})]>1-\epsilon/4,
\end{equation}
for any $w,g>0$, $1\le k_1, k_2 \le m_{w,g}$.
We next extend this bound to all boundary functions in $\sF_{w,n}$.

We note that for any $f_1, f_2, f_1', f_2' \in \sF_{w,n}$, 
\[
\max_{i\in\Z}|T_{f_1,f_2}(i)-T_{f_1',f_2'}(i)| \le \max_{|i|\le H_1n^{2/3}}|f_1(i)-f_1'(i)|
+ \max_{|i|\le H_1n^{2/3}}|f_2(i)-f_2'(i)|.
\]
Then by the second property of $\{f_{w,g,n;k}\}_{k=1}^{m_{w,g}}$ from above, for any $f_1, f_2 \in \sF_{w,n}$ we can find $f_{w,g,n;k_1}$ and $f_{w,g,n;k_2}$ such that
\[
\max_{i\in\Z}|T_{f_1,f_2}(i)-T_{f_{w,g,n;k_1},f_{w,g,n;k_2}}(i)| \le 2gn^{1/3}.
\]
Then $\NTP_{3g,n}(f_{w,g,n;k_1}, f_{w,g,n;k_2})$ implies $\NTP_{g,n}(f_1, f_2)$.
Therefore, by taking $g$ small enough (such that $3g < \delta$), we have
\begin{align*}
 & \lim_{n\to\infty} \inf_{f_1, f_2 \in \sF_{w,n}} \PP[\NTP_{g,n}(f_1, f_2)]  \\
 \ge & \lim_{n\to\infty} \inf_{1\le k_1, k_2\le m_w} \PP[\NTP_{\delta,n}(f_{w,g,n;k_1}, f_{w,g,n;k_2})] \\
 \ge & 1-\epsilon/4.
\end{align*}
Here the second inequality is due to \eqref{eq:pcov}. 
By plugging this into \eqref{eq:fotbd} the conclusion follows.
\end{proof}

\subsection{From closeness to coalescence}\label{ssec:close-to-coal}

We now deduce Proposition \ref{prop:stb} from Proposition \ref{prop:stb-weak}.
The approach involves taking a set of different $\alpha$ in Proposition \ref{prop:stb-weak}, and showing that it is unlikely for $\Gamma_{f_1,f_2}^n$ and $\Gamma_{f_1,f_3}^n$ to be close at their intersections with several $\LL_{\lfloor \alpha n\rfloor }$, while remaining disjoint (see Figure \ref{fig:46}).

We will need the following lemma, which says that any path passing through a given sequence of short segments is likely to have a small weight and thereby unlikely to be the geodesic.
Similar estimates have appeared in the literature in studying Exponential or other integrable LPP models. See e.g. \cite[Lemma 3.5]{BHS}, \cite[Theorem 1.9]{GH20}, \cite[Lemma 7.4]{MSZ}. 
\begin{lemma}  \label{lemma:stwtostb}
Given any $\epsilon>0$, any $M\in\N$ large enough (depending on $\epsilon$), and any $n\in\N$ large enough (depending on $\epsilon$ and $M$), the following is true.

Take any $\alpha_0 \in\R$, and let $\alpha_i = \alpha_0 + \frac{i}{10M}$ for each $i=1,\ldots,M$.
For any $r_0,\ldots,r_M \in \Z$ with each $|r_i| < M^{0.01}n^{2/3}$, we have
\[
\PP\left[ \max_{k_0,\ldots,k_M\in \llbracket 0, n^{2/3}/M^4\rrbracket}
\sum_{i=0}^{M-1}T^*_{w_i[r_i+k_i], w_{i+1}[r_{i+1}+k_{i+1}]} > \frac{4n}{10} - M^{0.05} n^{1/3}
\right] < \epsilon,
\]
where $w_i[k] = (\lfloor \alpha_in \rfloor + k, \lfloor \alpha_in \rfloor-k)$ for each $i\in \llbracket 1, M\rrbracket$ and $k\in\Z$. (Here recall that $T^*_{p,q} = T_{p,q}-\omega_p$ for any $p,q\in\Z^2$, $p\preceq q$.)
\end{lemma}
\begin{proof}
As before, in this proof we let $C,c>0$ denote large and small enough constants.

For each $i$, one has (for $n$ large enough)
\[
\PP\left[ \max_{k_i,k_{i+1}\in \llbracket 0, n^{2/3}/M^4\rrbracket}
T^*_{w_i[r_i+k_i], w_{i+1}[r_{i+1}+k_{i+1}]}
- T^*_{w_i[r_i], w_{i+1}[r_{i+1}]}
> M^{-1}n^{1/3}
\right] < C e^{-cM}.
\]
This is by Theorem \ref{thm:exp-to-dl}, which says that this probability converges to the corresponding probability in the directed landscape (uniformly in $r_i, r_{i+1}$), which is bounded using scaling and skew-shift invariance of the directed landscape (Lemma \ref{lem:DL-symmetry}), and the modulus of continuity of the directed landscape (Lemma \ref{lem:modcont}).
Next we take a union bound over all $0\leq i \leq M-1$, and take $M$ large, to conclude that
\[
\PP\left[ \sum_{i=0}^{M-1}\max_{k_i,k_{i+1}\in \llbracket 0, n^{2/3}/M^4\rrbracket}
T^*_{w_i[r_i+k_i], w_{i+1}[r_{i+1}+k_{i+1}]}
- T^*_{w_i[r_i], w_{i+1}[r_{i+1}]}
> n^{1/3}
\right] < \epsilon/2.
\]
We next estimate $T^*_{w_i[r_i], w_{i+1}[r_{i+1}]}$ for each $i$.
By the convergence to the directed landscape (Theorem \ref{thm:exp-to-dl}), the GUE Tracy-Widom marginal of the directed landscape, and that the GUE Tracy-Widom distribution has negative expectation \footnote{see \cite[Lemma A.4]{BGHH} for a proof by Ivan Corwin or \cite[Proposition 4.1]{GH20} for another proof.}, we have that for each $0\leq i \leq M-1$,
\[\E [T^*_{w_i[r_i], w_{i+1}[r_{i+1}]}] < \frac{4n}{10M} - c\left(\frac{n}{M}\right)^{1/3},\]
when $n$ is large enough.
Then since $T^*_{w_i[r_i], w_{i+1}[r_{i+1}]}$ are independent for all $i$ (due to that the weight of the first vertex $w_i[r_i]$ is removed from the passage time $T_{w_i[r_i], w_{i+1}[r_{i+1}]}$), using Theorem \ref{thm:lpp-onepoint}, we have that for large enough $n$,
\[
\Var\left(\sum_{i=0}^{M-1}T^*_{w_i[r_i], w_{i+1}[r_{i+1}]}\right)
=
\sum_{i=0}^{M-1}\Var(T^*_{w_i[r_i], w_{i+1}[r_{i+1}]}) < CM\left(\frac{n}{M}\right)^{2/3}.
\]
Thus by taking $M$ large, we have
\[
\PP\left[
\sum_{i=0}^{M-1}T^*_{w_i[r_i], w_{i+1}[r_{i+1}]} > \frac{4n}{10} - M^{0.1} n^{1/3}
\right] < \epsilon/2.
\]
and hence since for $M$ large enough  $-M^{0.1}+1<-M^{0.05}$, our conclusion follows.
\end{proof}

\begin{figure}[t]
         \centering
\begin{tikzpicture}
[line cap=round,line join=round,>=triangle 45,x=1cm,y=1cm]
\clip(-1,-1) rectangle (11,11);

\draw [line width=0.6pt] [blue] (1.6,1.4) -- (1.9,2.1) -- (2.3,2.7) -- (3.1,2.9) -- (3.6,3.4) -- (4.3,3.7) -- (4.4,4.6) -- (5.1,4.9) -- (5.4,5.6) -- (6,6) -- (6.2,6.8) -- (6.7,7.3) -- (7.3,7.7) -- (7.6,8.4) -- (8.3,8.7);
\draw [line width=0.3pt] [brown] (1.6,1.4) -- (1.9,2.1) -- (2.3,2.7) -- (3.1,2.9) -- (3.6,3.4) -- (4.3,3.7) -- (5.1,3.9) -- (5.4,4.6) -- (5.5,5.5) -- (6.1,5.9) -- (6.4,6.6) -- (7.2,6.8) -- (7.5,7.5) -- (7.9,8.1) -- (8.8,8.2);

\begin{scriptsize}

\draw (1,2) node[anchor=north east]{$f_1$};
\draw (8,9) node[anchor=south west]{$f_2$, $f_3$};
\draw (3.5,-0.5) node[anchor=north east]{$\LL_0$};
\draw (10.3,-0.3) node[anchor=north east]{$\LL_{\lfloor n/2\rfloor}$};
\draw (11,6) node[anchor=south east]{$\LL_n$};
\draw [blue] (7.6,8.4) node[anchor=south east]{$\Gamma^n_{f_1,f_2}$};
\draw [brown] (7.9,8.1) node[anchor=north west]{$\Gamma^n_{f_1,f_3}$};

\draw (5.4+2*0.6/5,5.6+2*0.4/5) node[anchor=east]{$p_0$};
\draw (5.5+2*0.6/5,5.5+2*0.4/5) node[anchor=north]{$q_0$};
\draw (6+4*0.2/5,6+4*0.8/5) node[anchor=south east]{$p_M$};
\draw (6.1+4*0.3/5,5.9+4*0.7/5) node[anchor=north west]{$q_M$};
\end{scriptsize}

\draw [line width=0.8pt] (0,3) -- (3,0);
\draw [line width=0.8pt] (10,7) -- (7,10);

\draw [dotted] [line width=0.5pt] (-2,5) -- (5,-2);
\draw [dotted] [line width=0.5pt] (12,5) -- (5,12);
\draw [dotted] [line width=0.5pt] (-2,12) -- (12,-2);

\foreach \i in {0,...,7}
{
\draw [dotted] [line width=0.5pt] (13.4+\i*0.2,-2) -- (-2,13.4+\i*0.2);
}

\foreach \i in {2,...,4}
{
\draw [fill=red, color=red] (5.4+\i*0.6/5,5.6+\i*0.4/5) circle (0.8pt);
\draw [fill=red, color=red] (5.5+\i*0.6/5,5.5+\i*0.4/5) circle (0.8pt);
}

\foreach \i in {0,...,4}
{
\draw [fill=red, color=red] (6+\i*0.2/5,6+\i*0.8/5) circle (0.8pt);
\draw [fill=red, color=red] (6.1+\i*0.3/5,5.9+\i*0.7/5) circle (0.8pt);
}

\end{tikzpicture}
        \caption{An illustration of the proof of Proposition \ref{prop:stb}: for the geodesics $\Gamma_{f_1,f_2}^n$ and $\Gamma_{f_1,f_3}^n$, if their intersections with each $\LL_{\lfloor \alpha_i n\rfloor}$ are close, it is unlikely for them to remain disjoint between $\LL_{\lfloor \alpha_0 n\rfloor}$ and $\LL_{\lfloor \alpha_M n\rfloor}$.}
        \label{fig:46}
\end{figure}
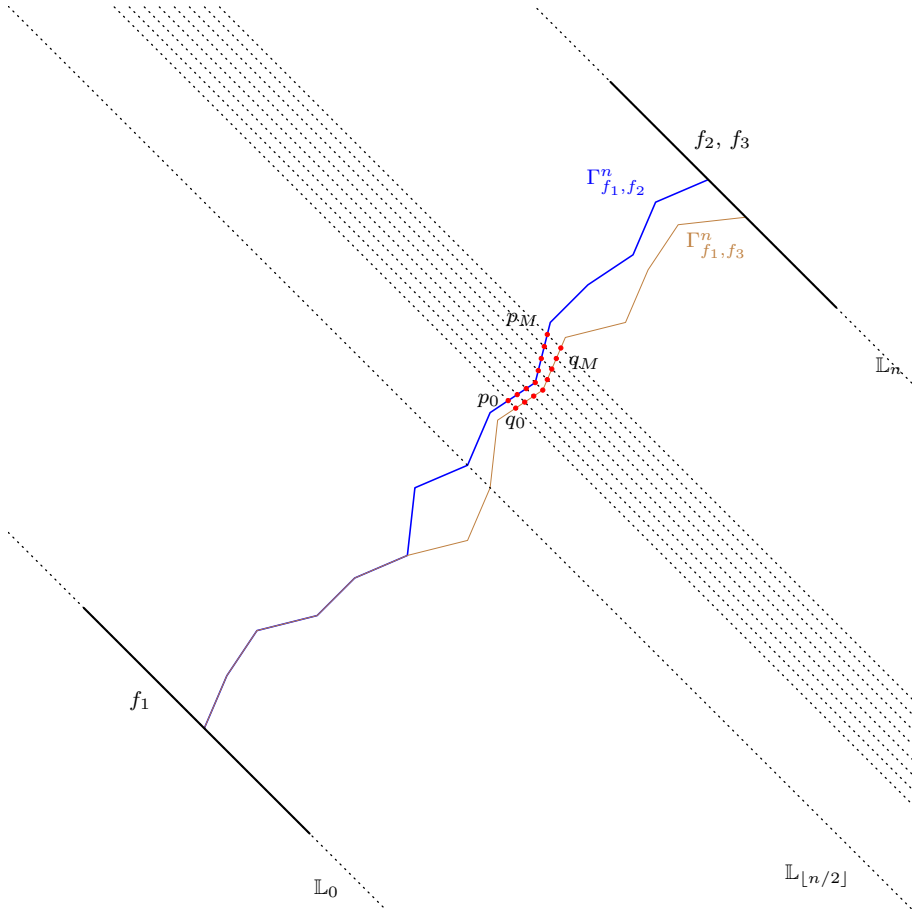

We now finish proving the stability result, by upgrading Proposition \ref{prop:stb-weak} using Lemma \ref{lemma:stwtostb}.

\begin{proof}[Proof of Proposition \ref{prop:stb}]
Let $C,c>0$ denote large and small enough constants in this proof.

We now explain the strategy.
From Proposition \ref{prop:stb-weak} we know that $\Gamma_{f_1,f_2}^n$ and $\Gamma_{f_1,f_3}^n$ are likely to stay close to each other at several given times.
To deduce coalescence from this, we will show that $\Gamma_{f_1,f_3}^n$ is unlikely to be competitive if it were disjoint from $\Gamma_{f_1,f_2}^n.$ To this end we will employ Lemma \ref{lemma:stwtostb}. However a key difference is that in the lemma the path is steered through fixed intervals whereas in the current situation one can view $\Gamma_{f_1,f_3}^n$ as passing through intervals adjacent to $\Gamma_{f_1,f_2}^n$ which are of course random. To address this we fix any upper-right path $\Gamma$ with two ending points on $\LL_0$ and $\LL_n$,
and condition on the event $\Gamma_{f_1,f_2}^n=\Gamma$. This ``fixes'' the intervals on one hand, but on the other hand, forces us to work with the conditioned environment. Fortunately, however, the FKG inequality comes to our aid and asserts that the above conditioning (being decreasing in nature) only makes the remainder of the environment  lower than what it typically is and hence  
the estimate from Lemma \ref{lemma:stwtostb} continues to hold true. 

Therefore we conclude that it is unlikely for $\Gamma_{f_1,f_3}^n$ to stay close to while being disjoint from $\Gamma$, given that $\Gamma_{f_1,f_2}^n=\Gamma$. Averaging over $\Gamma$ finishes the proof.

To implement the above plan, we first obtain closeness from Proposition \ref{prop:stb-weak}.
Let $M\in\N$ be a large constant to be determined.
Take $\alpha_0 = 0.6$, and for each $i=0,\cdots, M$, let $\alpha_i=\alpha_0+\frac{i}{10M}$.
Then each $\alpha_i$ is contained in $[0.6, 0.7]$, thus away from $0$ and $1$.
We let $p_i=(\lfloor \alpha_i n\rfloor + x_i, \lfloor \alpha_i n\rfloor - x_i)$, $q_i=(\lfloor \alpha_i n\rfloor + y_i, \lfloor \alpha_i n\rfloor - y_i)$ be the intersections of $\Gamma_{f_1,f_2}^n$ and $\Gamma_{f_1,f_3}^n$ with $\LL_{\lfloor \alpha_i n \rfloor}$, where $x_i, y_i \in \Z$.
By Proposition \ref{prop:stb-weak}, and taking a union bound over $i\in\llbracket 0, M\rrbracket$, we have
\begin{equation}  \label{eq:xydfbd}
\lim_{w, g\to 0}\limsup_{n\to\infty}\PP\Big[\max_{i\in\llbracket 0, M\rrbracket}|x_i-y_i|\ge n^{2/3}/M^4\Big] =0,    
\end{equation}
for any fixed $M\in \N$.

Next, we denote $\cE_M$ as the event where $|x_i-y_i| \le n^{2/3}/M^4, \forall 0\leq i \leq M$, and $\Gamma_{p_0, p_M}$, $\Gamma_{q_0, q_M}$ are disjoint.
Our goal is to show that, by taking $M$ large enough, we can make $\limsup_{n\to\infty}\PP[\cE_M]$ arbitrarily small.
To apply Lemma \ref{lemma:stwtostb} towards this, we need to bound each $x_i$ and $y_i$, and lower bound $T_{q_0,q_M}$.

\noindent\textbf{Estimate 1: transversal fluctuation.}
By Lemma \ref{lem:trans-dis-fini} we have
\[
\PP\left[\max\{|x_i|, |y_i| : 0\leq i \leq M \}> M^{0.01}n^{2/3}\right] < Ce^{-cM^{0.03}},
\]
which $\to 0$ as $M\to \infty$.\\
\noindent\textbf{Estimate 2: lower bound of $T_{q_0,q_M}$.}
Next, by splitting the sets $\{(\lfloor \alpha_0 n\rfloor +i, \lfloor \alpha_0 n\rfloor -i): i \in \llbracket -M^{0.01}n^{2/3}, M^{0.01}n^{2/3}\rrbracket\}$ and $\{(\lfloor \alpha_M n\rfloor +j, \lfloor \alpha_M n\rfloor -j): j \in \llbracket -M^{0.01}n^{2/3}, M^{0.01}n^{2/3}\rrbracket\}$ into segments of length in the order of $n^{2/3}$, and applying Proposition \ref{prop:seg-to-seg} to each pair, with the expectations of the passage times lower bounded using \eqref{e:mean} in Theorem \ref{thm:lpp-onepoint}, we have that
\[
\PP\left[\min_{i,j \in \llbracket -M^{0.01}n^{2/3}, M^{0.01}n^{2/3}\rrbracket} 
T_{(\lfloor \alpha_0 n\rfloor +i, \lfloor \alpha_0 n\rfloor -i), (\lfloor \alpha_M n\rfloor +j, \lfloor \alpha_M n\rfloor -j)}
< \frac{4n}{10} - CM^{0.02}n^{1/3} \right] < CM^{0.02}e^{-cM^{0.02}}
\]
for $n$ large enough.
From these we see that $\PP[ T_{q_0,q_M} < \frac{4n}{10} - CM^{0.02}n^{1/3}] < Ce^{-cM^{0.01}}$ for $n$ large enough.\\
\noindent\textbf{Estimate 3: staying close while disjoint makes $T_{q_0,q_M}$ small.}
We take any upper-right path $\Gamma$, with two ending points on $\LL_0$ and $\LL_n$.
Take $r_i = \Z$ such that
$(\lfloor \alpha_i n\rfloor + r_i, \lfloor \alpha_i n\rfloor - r_i)$ is the intersection of $\Gamma$ with $\LL_{\lfloor \alpha_i n\rfloor}$, and assume that each $|r_i|<M^{0.01}n^{2/3}$.
Let $\cE_\Gamma$ be the event where there exist $k_0,\ldots,k_M\in \llbracket -m^{2/3}/M^4, m^{2/3}/M^4\rrbracket$, such that there is an upper-right path from $(\lfloor \alpha_0 n\rfloor +r_0+k_0, \lfloor \alpha_0 n\rfloor -r_0-k_0)$ to $(\lfloor \alpha_M n\rfloor +r_M+k_M, \lfloor \alpha_M n\rfloor -r_M-k_M)$, passing through each $(\lfloor \alpha_i n\rfloor +r_i+k_i, \lfloor \alpha_i n\rfloor -r_i-k_i)$, being disjoint from $\Gamma$, and having total weight $>\frac{4n}{10} - CM^{0.02}n^{1/3}$.
By the FKG inequality, one has $\PP[\cE_{\Gamma} \mid \Gamma_{f_1,f_2}^n=\Gamma]\leq \PP[\cE_\Gamma]$,
since $\cE_{\Gamma}$ is a positive event for the random field $\omega$ on $\Z^2\setminus\Gamma$, while $\Gamma_{f_1,f_2}^n=\Gamma$ is a negative event for the random field outside $\Gamma$.
By Lemma \ref{lemma:stwtostb} we can make   $\PP[\cE_{\Gamma}]$ small uniformly in $\Gamma$, by taking $M$ and then $n$ large enough.

Note that
\begin{align*}
\PP[\cE_M] \le &\PP\left[\max\{|x_i|, |y_i| : 0\leq i \leq M \}> M^{0.01}n^{2/3}\right] \\
&+ \PP[ T_{q_0,q_M} < \frac{4n}{10} - CM^{0.02}n^{1/3}] \\ &+ \sum_{\Gamma}\PP[\cE_{\Gamma}]\PP[\Gamma_{f_1,f_2}^n=\Gamma]    
\end{align*}
where the sum is over all $\Gamma$ satisfying the above conditions. 
By the three estimates above, we conclude that $\PP[\cE_M]$ could be made arbitrarily small by taking $M$ and then $n$ large enough.

Finally, from the definition of $\cE_M$ we have
\[
\PP[\Gamma_{p_0,p_M}\cap\Gamma_{q_0,q_M}=\emptyset] \le \PP[\cE_M] + \PP\Big[\max_{i\in\llbracket 0, M\rrbracket}|x_i-y_i|\ge n^{2/3}/M^4\Big].
\]
Using \eqref{eq:xydfbd} and the above bound of $\PP[\cE_M]$, we can make $\PP[\Gamma_{p_0,p_M}\cap\Gamma_{q_0,q_M}=\emptyset]$ arbitrarily small, by taking $M$ large enough, then $w,g$ small enough, and finally $n$ large enough.
Noting that $\Gamma_{p_0,p_M}\cap\Gamma_{q_0,q_M}\neq\emptyset$ would imply $\Gamma_{f_1,f_2}^n \cap \LL_{\lfloor n/2\rfloor} = \Gamma_{f_1,f_3}^n \cap \LL_{\lfloor n/2\rfloor}$, the conclusion follows.
\end{proof}

\bibliography{bibliography}

\appendix

\section{Maximum occupation time of a segment}   \label{ssec:max-ocp-seg}

In this appendix we prove Lemma \ref{lem:disc-max}, using arguments similar to those appearing in \cite[Section 3]{SSZ}.
For each $a<b\in\Z,$ we denote
\[
W'_{a,b}=\max_{2a-b\le a' \le a, b'\ge b} W[a',b'] \vee \max_{a' \le a, b\le b'\le 2b-a} W[a',b'].
\]
Note that one always has $W'_{a,b}\le W_{a,b}$.
\begin{lemma}  \label{lem:disc-max-out}
We have $\E[W_{a,b}']  < C (b-a)^{1/3}$, where $C>0$ is a universal constant.
\end{lemma}

\begin{proof}
Without loss of generality, we assume that $a=0$.
We take a dyadic decomposition. Namely, by Lemma \ref{lem:oc-one-pt} below, for any integer $0\le m\le b/2$ we have
\begin{multline*}
\PP[ (m,m) \in \cup_{-b\le i \le 0, j \ge b} \Gamma_{(i,i),(j,j)}] \\< \sum_{k=0}^{\lfloor\log_2(b/m+1)\rfloor} \PP[ (m,m) \in \cup_{m-2^{k+1}m\le i \le m-2^km, j \ge b} \Gamma_{(i,i),(j,j)}] < Cm^{-2/3},    
\end{multline*}
and
\[
\PP[ (m,m) \in \cup_{i \le -b, b\le j \le 2b} \Gamma_{(i,i),(j,j)}] < Cb^{-2/3}.
\]
Here $C>0$ is a universal constant.
By symmetry, we have similar bounds when $b/2\le m \le b$.
Thus by summing over $0\le m \le b$ we get the conclusion. 
\end{proof}

\begin{lemma}  \label{lem:oc-one-pt}
For any $m\in\N$ we have $\PP[(m,m)\in \cup_{-m\le i \le 0, j \ge 2m} \Gamma_{(i,i),(j,j)} ] < C m^{-2/3}$, where $C>0$ is a universal constant.
\end{lemma}
This estimate is easier (than Lemmas \ref{lem:disc-max} and \ref{lem:disc-max-out}) because here $(m,m)$ is away from the endpoints of the considered geodesics $\Gamma_{(i,i),(j,j)}$, allowing for using translation invariance arguments.
\begin{proof}[Proof of Lemma \ref{lem:oc-one-pt}]
By translation invariance, we have that
\[
\PP[(m,m)\in \cup_{-m\le i \le 0, j \ge 2m} \Gamma_{(i,i),(j,j)} ] \le m^{-2/3} \E[X],
\]
where \[X=|\{ |k| \le m^{2/3}:  (m+k, m-k) \in \cup_{-m\le i \le 0, |i'| \le m^{2/3}, j' \in \Z} \Gamma_{(i+i', i-i'), (2m+j',2m-j')}  \}|.\]

Take $1<M\le 2m^{2/3}+1$, we now bound $\PP[X>M]$.
Let $\ell = M^{1/4}$.
Let $\cE_1$ be the event where the geodesic $\Gamma_{(-m + \lfloor 2\ell m^{2/3}\rfloor, -m - \lfloor 2\ell m^{2/3}\rfloor), (m + \lfloor 2\ell m^{2/3}\rfloor, m - \lfloor 2\ell m^{2/3}\rfloor)}$
always stays between the lines $\{(x + \ell m^{2/3}, x-\ell m^{2/3}): x\in \R\}$ and $\{(x + 3\ell m^{2/3}, x-3\ell m^{2/3}): x\in \R\}$;
and the geodesic $\Gamma_{(-m - \lfloor 2\ell m^{2/3}\rfloor, -m + \lfloor 2\ell m^{2/3}\rfloor), (m - \lfloor 2\ell m^{2/3}\rfloor, m + \lfloor 2\ell m^{2/3}\rfloor)}$ always stays between the lines $\{(x - \ell m^{2/3}, x+\ell m^{2/3}): x\in \R\}$ and $\{(x - 3\ell m^{2/3}, x+3\ell m^{2/3}): x\in \R\}$.
By Lemma \ref{lem:trans-dis-fini} we have $\PP[\cE_1] > 1-Ce^{-c\ell^3}$ for some constants $c,C>0$.

We let $\cE_2$ be the event where
\[
|\cup_{|i|<3\ell m^{2/3}, j\in\Z} \Gamma_{(i,-i), (2m+j,2m-j)} \cap \{(m+k, m-k): k \in \llbracket -m^{2/3}, m^{2/3}\rrbracket\}| \le  M.
\]
By Lemma \ref{lem:num-int-dis}, we have $\PP[\cE_2] > 1-Ce^{-cM^{1/12800}}$ for some constants $c,C>0$.

We note that under $\cE_1$, any geodesic $\Gamma_{(i+i', i-i'), (2m+j',2m-j')}$ for $-m\le i \le 0, |i'| \le m^{2/3}, j' \in \Z$ must always stay between the lines $\{(x - 3\ell m^{2/3}, x+3\ell m^{2/3}): x\in \R\}$ and $\{(x + 3\ell m^{2/3}, x-3\ell m^{2/3}): x\in \R\}$; so under $\cE_1\cap \cE_2$ we must have $X\le M$.
Then we have that $\E[X]$ is bounded by a constant (independent of $m$), so the conclusion follows.
\end{proof}

We finish proving Lemma \ref{lem:disc-max} using Lemma \ref{lem:disc-max-out}, and another dyadic decomposition.
\begin{proof}[Proof of Lemma \ref{lem:disc-max}]
Without loss of generality, we always assume that $a=0$.

We prove by induction in $b$.
The case where $b<A$ is obvious, so we assume that $b \ge A$, and that \eqref{eq:dis-W-exp} and \eqref{eq:dis-W-tail} hold for all smaller $b$.

We first prove \eqref{eq:dis-W-exp}.
Take $g$ to be the largest integer such that $2^g\le b/A$.
Let $\Lambda=\{(m2^k, (m+1)2^k): k\ge g, m\in\Z, 0\le m2^k\le (m+1)2^k\le b\}$, and $\partial\Lambda=\{(m2^g, (m+1)2^g): m\in\Z, 0\le m2^g\le (m+1)2^g\le b\}$.
We then have that
\[
W_{0,b} \le \sum_{(i,j)\in \Lambda} W_{i,j}'+2\max_{(i,j)\in \partial\Lambda} W_{i,j}.
\]
By Lemma \ref{lem:disc-max-out} we have
\[
\E[W_{0,b}] \le C_1 b2^{-2g/3} + 2\E[\max_{(i,j)\in \partial\Lambda} W_{i,j}],
\]
where $C_1>0$ is a universal constant.
By induction hypothesis, for each $(i,j)\in \partial\Lambda$ and $M>0$, we have
\[
\PP[W_{i,j} >AM2^{g/3}] < Ce^{-cM}.
\]
Thus we have
\[
\E[\max_{(i,j)\in \partial\Lambda}W_{i,j}] < c^{-1}(\log(2^{-g}b)+C)A 2^{g/3},
\]
and this implies that 
\[
\E[W_{0,b}]  < C_1 b2^{-2g/3} + c^{-1}(\log(2^{-g}b)+C)A 2^{g/3} < (2C_1 + c^{-1}(\log(2A)+C))A^{2/3}b^{1/3}
< A^{2/3+0.01}b^{1/3},
\]
where the last inequality is by requiring $A$ being large enough so that $A^{0.01}>2C_1 + c^{-1}(\log(2A)+C)$.

We next prove \eqref{eq:dis-W-tail}. Denote $r=\lfloor A^{1/10} \rfloor$.
For each $1\le i \le r$ we denote \[W_i=W_{\lfloor(i-1)(b+1)/r \rfloor,\lfloor i(b+1)/r\rfloor-1},\] and we then have $W_{0,b}\le \sum_{i=1}^r W_i$.
Thus
\[
\sum_{i=1}^r W_i \le \sum_{i=1}^r W_i\don[W_i\le Ab^{1/3}/r] + \sum_{i=1}^r W_i\don[W_i> Ab^{1/3}/r] \le Ab^{1/3} + \sum_{i=1}^r W_i\don[W_i> Ab^{1/3}/r].
\]
Denote $X_i = (A(b/r+2))^{-1/3}W_i\don[W_i> Ab^{1/3}/r]$. We also have that (from induction hypothesis) $\E[W_i]\le A^{2/3+0.01}(b/r+2)^{1/3}$; so for any $M>0$ by induction hypothesis we have
\[
\PP[X_i>M] < Ce^{-cM} \wedge 2A^{-1/3+0.01}r^{2/3}.
\]
Take $M_*>0$ such that $Ce^{-cM_*} = A^{-1/3+0.01}r^{2/3}$.
Take $a=c/10$. We then have
\[
\E[e^{aX_i}] = 1 + \int ae^{aM}\PP[X_i>M] dM
< 1 + 4e^{aM_*} A^{-1/3+0.01}r^{2/3},
\]
and
\[
\E[e^{a\sum_{i=1}^rX_i}] < (1 + 4e^{aM_*} A^{-1/3+0.01}r^{2/3})^r
<
e^{C(A^{-1/3+0.01}r^{2/3})^{1/2}r},
\]
where the last inequality is by 
\[
4e^{aM_*} A^{-1/3+0.01}r^{2/3} < C(A^{-1/3+0.01}r^{2/3})^{1/2},
\]
which is true by taking $C>10$.
Then 
\[
\PP[\sum_{i=1}^rX_i > r^{1/3}(M-1)/2] < e^{C(A^{-1/3+0.01}r^{2/3})^{1/2}r - ar^{1/3}(M-1)/2} < Ce^{-cM},
\]
where the last inequality is true by taking $A$ large enough so that $ar^{1/3}>4c$, and \\
$e^{C(A^{-1/3+0.01}r^{2/3})^{1/2}r} < C$, and assuming that $M>2$.
Note that $\sum_{i=1}^rX_i \le r^{1/3}(M-1)/2$ implies that $W_{0,b}\le \sum_{i=1}^r W_i \le AMb^{1/3}$, we have that the conclusion follows.
\end{proof}

\end{document}